\documentclass[11pt]{article}
\usepackage{graphicx}
\usepackage{xcolor}
\usepackage[scale=0.8]{geometry}
\usepackage{amsmath,amssymb,amsthm,graphicx,float,url}
\usepackage{dsfont}
 \usepackage{mathrsfs}
\usepackage{mathtools}
\usepackage{upgreek}
\usepackage{makecell}
\usepackage{dirtytalk}
\usepackage{tikz,mathdots,yhmath,cancel,siunitx,array,multirow,amssymb,gensymb,tabularx,extarrows,booktabs}
\usetikzlibrary{fadings}
\usetikzlibrary{patterns}
\usetikzlibrary{shadows.blur}
\usetikzlibrary{shapes}
\usepackage{pgfplots,graphicx,float,svg}
\usepackage{caption} 
\usepackage{algorithm}
\usepackage{algorithmic}

\usepackage{comment}
\usepackage{subfigure}
\usepackage{caption} 
 \usepackage{mathtools}
 \usepackage{bm}
 \usepackage{enumerate}
 
\usepackage[colorlinks=true]{hyperref}
\usepackage[nameinlink,capitalize,noabbrev]{cleveref}
\crefname{section}{Section}{Sections}
\crefname{subsection}{Section}{Sections}
\crefname{subsubsection}{Section}{Sections}

\usepackage{pdfsync}
\usepackage{float}
\usepackage{mathrsfs}
\title{\sf Hearing the Shape of a Cuboid Room Using Sparse Measure Recovery}
\author{Antoine Deleforge\footnote{Inria, IRMA, Universit\'e de Strasbourg, CNRS UMR 7501, 67084 Strasbourg, France. ({\tt antoine.deleforge@inria.fr})}
\and C\'edric Foy\footnote{UMRAE, Cerema, Univ. Gustave Eiffel, Ifsttar, Strasbourg, 67035, France. ({\tt cedric.foy@cerema.fr})}
\and Yannick Privat\footnote{Universit\'e de Lorraine, CNRS, Institut Elie Cartan de Lorraine, Inria, BP 70239 54506
Vandœuvre-l\`es-Nancy Cedex, France. ({\tt yannick.privat@univ-lorraine.fr}).}~\footnote{Institut Universitaire de France (IUF)}
\and  Tom Sprunck.\footnote{Inria, IRMA, Universit\'e de Strasbourg, CNRS UMR 7501, 67084 Strasbourg, France. ({\tt tom.sprunck@cea.fr})}~\footnote{Corresponding author.}
}

\date{\today}

\newcommand{\avect}{\bm{a}}
\newcommand{\bvect}{\bm{b}}

\newcommand{\evect}{\bm{e}}

\newcommand{\jvect}{\bm{j}}

\newcommand{\nvect}{\bm{n}}

\newcommand{\qvect}{\bm{q}}
\newcommand{\rvect}{\bm{r}}

\newcommand{\uvect}{\bm{u}}
\newcommand{\vvect}{\bm{v}}

\newcommand{\xvect}{\bm{x}}

\newcommand{\alphavect}{\bm{\alpha}}

\newcommand{\varepsilonvect}{\bm{\varepsilon}}

\def\R{\textrm{I\kern-0.21emR}}
\def\N{\textrm{I\kern-0.21emN}}
\def\Z{\mathbb{Z}}

\newcommand{\src}{\rvect^\text{src}}
\renewcommand{\geq}{\geqslant}
\renewcommand{\leq}{\leqslant}
\newcommand{\Reps}{\R^3_\varepsilon}
\newcommand{\Oset}[2]{\mathcal O^{#2}_{#1}}
\newcommand{\mic}{\rvect^\text{mic}}

\newcommand{\dint}[2]{[\![#1 ,#2]\!]}
\newcommand{\M}{\mathcal{M}}
\newcommand{\tvn}[1]{\Vert #1 \Vert_\text{TV}}
\newcommand{\infn}[1]{\Vert #1 \Vert_\infty}
\newcommand{\Ltn}[1]{\left\lVert #1 \right\rVert _2}

\newcommand{\Lon}[1]{\left\lVert #1 \right\rVert _1}

\renewcommand{\geq}{\geqslant}
\renewcommand{\leq}{\leqslant}

\newtheorem{theorem}{Theorem}  
\newtheorem{proposition}{Proposition}
\newtheorem{corollary}{Corollary}
\newtheorem{definition}{Definition}
\newtheorem{lemma}{Lemma}

\newtheorem{remark}[theorem]{Remark} 

\newcommand\ts[1]{{#1}}
 
\begin{document}

\maketitle

\begin{abstract}
This article explores a variant of Kac's famous problem, ‘‘Can one hear the shape of a drum?'', by addressing a geometric inverse problem in acoustics. Our objective is to reconstruct the shape of a cuboid room using acoustic signals measured by microphones placed within the room. By examining this straightforward configuration, we aim to understand the relationship between the acoustic signals propagating in a room and its geometry. This geometric problem can be reduced to locating a finite set of acoustic point sources, known as image sources. We model this issue as a finite-dimensional optimization problem and propose a solution algorithm inspired by super-resolution techniques. This involves a convex relaxation of the finite-dimensional problem to an infinite-dimensional subspace of Radon measures. We provide analytical insights into this problem and demonstrate the efficiency of the algorithm through multiple numerical examples.
\end{abstract}



\section{Introduction}

\subsection{From acoustic measurements to geometric reconstruction}

In his seminal article \cite{kac1966can}, Kac explored the connection between a drum's sound and its geometry. He formulated this question as a spectral geometry problem, and investigated the existence of a one-to-one relationship between a bounded domain and the spectrum of its associated Laplace-Dirichlet operator. In this article, we address a tangential question notably arising in room acoustics, namely, exploring the relationship between the geometry of a 3D enclosure and acoustic signals recorded within this enclosure. In contrast to Kac's formulation, a number of physical constraints intrinsic to acoustic measurements are considered, namely, signals can only be measured within a limited frequency band, over a finite time, at a finite number of spatial locations, and are subject to noise.

The finiteness of measurements makes the general problem of recovering an arbitrary room shape -- an infinite dimensional object -- obviously ill-posed. {To make the unknown finite-dimensional,} we narrow the focus here to \textit{cuboid} rooms, whose shape can be fully described by their length, width and height. We further assume that a point source at an unknown location in the room emits an impulse a time $t=0$, and that the relative positions of measurement microphones are known but not their absolute positions and orientation within the room.

In {the cuboid setting}, the room-shape recovery problem can be recast as that of detecting and localizing a sufficiently large number of so-called \textit{image sources}, that lie outside of the room's boundary. Indeed, an efficient algorithm to recover a cuboid room's geometrical parameters from such an image-source \textit{point cloud} was recently proposed by the authors in \cite{sprunck2024fullreversing}. Hence, this article entirely focuses on the image-source recovery problem itself\footnote{A preliminary study on this problem was recently published by the authors in a short paper \cite{sprunck2022gridless}.}.
%
%
We will explain why adopting this viewpoint amounts to considering a wave equation where the source term is an infinite combination of Dirac measures, each localized at the position $\rvect_k$ of an image source.
In other words, we consider an equation of the form:
\begin{equation}\label{mainEq0}
\frac{1}{c^2}\partial_{tt}p(t,\rvect)-\Delta p(t,\rvect) = \sum_{k=0}^{+\infty} a_k\delta(\rvect-\rvect_k)\delta (t),\qquad (t,\rvect)\in \R_+\times \R^3,
\end{equation}
where p denotes the sound pressure, $c>0$ the acoustic-wave speed, the $a_k$'s are real coefficients, {$\delta(\rvect-\rvect_k)$} denotes the Dirac measure in space at the point $r=\rvect_k$ and {$\delta(t)$} the Dirac measure in time at $t=0$.
For this equation, estimating the source term from pressure measurements amounts to determining both the coefficient and pointwise support of all Dirac measures, which constitutes a challenging, high-dimensional, non-linear, non-convex, inverse problem.


In this article, we aim to achieve several objectives:
\begin{itemize}
\item Model the problem;
\item Analyze the underlying optimization challenges;
\item Develop and propose an efficient resolution algorithm.
\end{itemize}
%
We will show that the problem cannot be easily reformulated as an observability inequality due the following physical constraints:
\begin{itemize}
\item Microphones act as low-pass filters, capturing only discrete low-frequency sound pressure measurements;
\item Measured signals are typically noisy, meaning that what is measured is not the exact solution to the wave equation.
\end{itemize}
To address these issues, we propose modeling the reconstruction problem as an optimization task over the set of Radon measures, incorporating a regularization term that promotes sparsity.

\subsection{State of the art}
The inverse problem of recovering image source locations and amplitudes from measured signals has been the subject of a substantial amount of contributions from the audio signal processing community. The aim is usually to locate reflectors, such as walls, by finding the true source position along with the associated first-order image sources. 
Most approaches consist of three steps: estimating the times and/or directions of arrival of image sources, labeling them, and finally apply a triangulation to recover the location \cite{tervo2010estimation,antonacci2012inference,sun2012localization,mabande2013room,dokmanic2013acoustic,jager2016room,remaggi2016acoustic,el20173d,lovedee2019three, macwilliam2023simultaneous}. 
Alternatively, \cite{ribeiro2011geometrically} proposes a more direct approach that directly operates in 3D space. Retrieving the coefficients $a_k$ associated to reflectors in frequency bands is studied in \cite{shlomo2021blind} and \cite{dilungana2022geometry}, as they relate to their acoustic \textit{impedance}. Finally, recovering image sources within a given range is the focus of recent non-parametric sound-field reconstruction methods \cite{koyama2019sparse, damiano2021soundfield}. 

Most of these references define a discretization in time \cite{tervo2010estimation,sun2012localization,mabande2013room,antonacci2012inference, dokmanic2013acoustic,kowalczyk2013blind, crocco2016estimation, jager2016room,remaggi2016acoustic,el20173d,lovedee2019three,shlomo2021blindloc,dilungana2022geometry}, in 2D space \cite{sun2012localization,mabande2013room,remaggi2018acoustic,koyama2019sparse,damiano2021soundfield} or in 3D space \cite{ribeiro2011geometrically}, and apply peak-picking techniques and/or sparse optimization methods over a finite grid. This approach is flawed by computational constraints in 3D, as the grid size grows cubically with the resolution and limits the accuracy of sparse methods \cite{ribeiro2011geometrically,koyama2019sparse,damiano2021soundfield}. Meanwhile, peak-picking methods for times of arrival estimation often fail when echoes are overlapping and the signal is degraded. This issue is countered in most cases by employing additional assumptions on the source and microphone positions within the room \cite{antonacci2012inference,dokmanic2013acoustic,jager2016room,remaggi2016acoustic,el20173d,lovedee2019three}, which is unsatisfactory in the context of room geometry inference. Finally, sparse optimization over a discrete grid is hindered by the so-called \textit{basis-mismatch} problem \cite{chi2011sensitivity}, which severely reduces the performance of reconstruction algorithms.

Outside this community, the gridless spike recovery problem has been vastly studied theoretically and numerically \cite{de2012exact,candes2014towards,duval2015exact, morgenshtern2016super,denoyelle,traonmilin2020basins}, with notable applications to \textit{super-resolution} in, \textit{e.g.}, fluorescence microscopy \cite{huang2009super, denoyelle}.
The theoretical gridless spike recovery problem can be approached either by Prony's method and its derivatives, such as MUSIC \cite{schmidt1986multiple} (MUltiple SIgnal Classification), ESPRIT \cite{roy1989esprit} (Estimation of Signal Parameters by Rotational Invariance Technique), or by variational methods. While some extensions to noisy data \cite{condat2015cadzow, wagner2021gridless} and multivariate measures \cite{peter2015prony,kunis2016multivariate} were developed, Prony's type methods are better suited to noiseless, 1D measurements.
On the other hand, gridless variational methods aim to resolve optimization problems on a space of measures, without prior knowledge on the number of spikes. These problems can be seen as convex relaxations of similar finite dimension, grid-based problems to infinite dimensional convex optimization problems.
We will focus in our numerical applications on these variational methods, which generalize well to any kind of measurement operator and noise. More precisely, we will consider the relaxed BLASSO optimization problem (\hyperref[eq:relaxed_blasso]{{\mbox{$\mathscr{B}_{\lambda}$}}}), 
which has received considerable attention in recent years, both in theoretical and algorithmic work.
A first theoretical issue is to find conditions to ensure exact support recovery of the ground truth measure in the noiseless case. It then follows naturally to study how the solution of (\hyperref[eq:relaxed_blasso]{{\mbox{$\mathscr{B}_{\lambda}$}}}) with noisy measurements relates to the solution of the noiseless case.
The seminal work of Candès and Fernandes-Granda \cite{candes2014towards} on the 1D low-pass filter vastly contributed to opening the field by proving that exact support recovery could be achieved under a minimum separation constraint between spikes, with latter expansions to noisy measurements \cite{candes2013super, azais2015spike, fernandez2013support}. These latter works provide error bounds on the locations of the recovered spikes, but few guarantees on the structure of the reconstructed measure. Duval and Peyré show in \cite{duval2015exact} that under certain hypotheses on the certificates of the dual problem, the regularization parameter $\lambda$ and the measurement noise, there exists a unique solution to the noisy BLASSO  that contains as many spikes as the input measure. Additionally, the recovered measure converges to the exact measure for the weak-$\ast$ topology when the noise and $\lambda$ vanish to $0$.
Note that a particular case of interest focuses on input measures for which the spikes cluster around a set point. Exact noiseless support recovery and stable noisy reconstruction can be achieved when the amplitudes of the spikes are real and positive under some non-degeneracy condition, see for instance \cite{denoyelle2017support} in 1D or \cite{poon2019multidimensional} for higher dimensions, however in our case the spikes will be well-separated in space. 

Several successful numerical approaches have been developed over the years in order to solve the super-resolution problem off the grid. We can cite amongst these methods the semi-definite programming (SDP) formulation \cite{candes2014towards} and its extension to higher dimensions using Lasserre hierarchy \cite{lasserre2009moments, de2016exact}, optimal transport theory and particle gradient descent \cite{chizat2018global, chizat2022sparse}, over-parametrized projected gradient descent \cite{traonmilin2020projected, benard2022fast} and finally the Frank-Wolfe algorithm (also called the conditional gradient descent) \cite{frank1956algorithm,bredies2013inverse,denoyelle, boyd2017alternating}.


\subsection{Main contributions and organization of the article}

Section~\ref{sect:model} is dedicated to the mathematical formalization of the inverse problem at hand. In Section~\ref{sec:directModel}, we explain how to explicitly solve the wave equation in our setting, and in Section~\ref{sec:modelFiniteD}, we reformulate the problem as a finite-dimensional optimization problem. We will partially analyze this problem in Section~\ref{sec:analysis}, revealing that the problem can be either well-posed or ill-posed depending on the collected data. This insight leads us to consider a slightly modified formulation in Section~\ref{sect:IS_inv_theory} to avoid certain pathologies related to the problem's unique characteristics (particularly the singularities of Green's kernels).

Interestingly, and somewhat paradoxically, we will develop a numerical method based on an equivalent formulation of the optimization problem in infinite dimensions, introduced in in Section~\ref{sect:IPdimInf}. In this approach, the unknowns are the positions of the image sources $\rvect_k$ and the attenuation coefficients $a_k$ (see Equation~\ref{mainEq0}) of the acoustic signal. This method has the advantage of convexifying the problem and will be detailed in Section~\ref{sect:numAlgo}. Section~\ref{sect:IS_inv_exp} is dedicated to numerical experiments that demonstrate the effectiveness of the proposed algorithm.

%
\section{Modeling of the inverse problem}
\label{sect:model}

\subsection{The direct problem}\label{sec:directModel}

Consider a rectangular room $\Omega$ with (positive) dimensions $L_x$, $L_y$, $L_z$, and a sound source positioned at $\src$ within the room. Let $\mic\neq \src$ denote a microphone location distinct from the source position. The {\it Room Impulse Response} (RIR) for this configuration is the signal recorded at microphone $\mic$ when the source emits an ideal impulse at time $t=0$. In other words, a RIR represents the measurement at a given microphone location of the Green's function for the wave equation. A \emph{multi-channel} RIR refers to a collection of RIRs recorded at various microphone locations for a single source position.

The pressure field $p$ resulting from a {\it perfectly impulsive source} located at $ \src$ is a solution to the inhomogeneous wave equation \eqref{eq:eq:wave_neu}:
\begin{equation}
    \left\{\label{eq:eq:wave_neu}
    \begin{array}{ll}
        \frac{1}{c^2}\partial_t^2 p(\rvect,t) - \Delta p(\rvect,t) = \delta(t)\delta(\rvect-\src)& (\rvect,t)\in\Omega\times\R \\
        p(\rvect,t) = \partial_t p (\rvect,t) = 0 & (\rvect,t)\in\Omega\times \R_-^* ,
    \end{array}\right.  
\end{equation}
where $c>0$ is the speed of sound.

The partial absorption and reflection of sound waves at the walls are typically modeled by incorporating admittance boundary conditions on $\partial\Omega$ \cite{bruneau2013fundamentals}: 
\begin{equation}\label{eq:bc_time}
    \partial_{\nvect}p(\rvect,t) + \frac{1}{c}\frac{\partial}{\partial t}\beta(\rvect, \cdot)*p(\rvect,\cdot)(t) = 0 \quad (\rvect,t) \in \partial\Omega\times\R,
\end{equation}
where $\beta$ is the time-domain admittance of the wall and $*$ denotes time-domain convolution. 
In the following, we will consider a modified version of the simplified ideal case with perfectly reflecting walls, which corresponds to setting a constant $\beta(\cdot)$ in Eq.~\eqref{eq:bc_time}, \emph{i.e.}, applying Neumann boundary conditions on $\partial\Omega$. 
The pressure field $p$ is then solution to the following system:
\begin{equation}    \label{eq:wave_neu}
    \left\{
    \begin{array}{ll}
        \frac{1}{c^2}\partial_t^2 p(\rvect,t) - \Delta p(\rvect,t) = \delta(t)\delta(\rvect-\src)& (\rvect,t)\in\Omega\times\R \\
        p(\rvect,t) = \partial_t p (\rvect,t) = 0 & (\rvect,t)\in\Omega\times \R_-^* \\
        \partial_{\nvect }p(\rvect,t) = 0 & (\rvect,t)\in\partial \Omega\times\R. \\
    \end{array}\right .
\end{equation}

\begin{remark}[Equivalence between two formulations of the wave equation: one with a source term and no initial velocity, and the other without a source term but with an initial velocity]
Let us set $\square_c= \frac{1}{c^2}\partial_t^2 - \Delta$ and let $\varphi_1$ denote a smooth function in $\Omega$. The function $p$ is the {\it fundamental solution}, also known as {\it Green kernel} or {\it Green's function} of the wave equation on $\Omega\times \R$ with Neumann boundary conditions. Let us highlight that $p$ also solves an equivalent formulation.
To this aim, we consider $\varphi$, the unique solution of the wave equation  
\begin{equation}
    \left\{
    \begin{array}{ll}
        \square_c \varphi (\rvect,t) = 0 & (\rvect,t)\in\Omega\times\R_+ \\
         \varphi (\rvect,0) =0, \quad  \partial_t  \varphi  (\rvect,0) = \varphi_1 (\rvect) & \rvect \in\Omega \\
        \partial_{\nvect } \varphi (\rvect,t) = 0 & (\rvect,t)\in\partial \Omega\times\R_+ .
    \end{array}\right .
\end{equation}
Let us introduce the function $\psi$ defined by $\psi(\rvect,t)=H(t)\varphi(\rvect,t)$, where $H$ is the so-called Heaviside function. Then, seeing $\psi$ as a distribution, one gets 
\begin{equation}
    \square_c \psi (\rvect,t) =\delta (t)\varphi_1(\rvect), \qquad (\rvect,t)\in\Omega\times\R
\end{equation}
according to the jump rule.
Let $\widetilde p $ the solution of 
\begin{equation}
    \label{eq:wave_neu2}
    \left\{
    \begin{array}{ll}
        \square_c \widetilde p(\rvect,t) = 0 & (\rvect,t)\in\Omega\times\R_+ \\
        \widetilde p(\rvect,0) =0, \quad  \partial_t \widetilde p (\rvect,0) = \delta(\rvect-\src) & \rvect \in\Omega \\
        \partial_{\nvect }\widetilde p(\rvect,t) = 0 & (\rvect,t)\in\partial \Omega\times\R_+ .
    \end{array}\right .
\end{equation}
Using the representation formula through Green kernels and denoting by $\ast_{\rvect}$ the spatial convolution, for any choice of spatial source term $\varphi_1$ the equality $p\ast_{\rvect} \varphi_1=\widetilde {p}\ast_{\rvect} \varphi_1$ stands for positive times, thus $p$ is also solution to System \eqref{eq:wave_neu2}.
\end{remark}

It is notable that the analytical solution to the system \eqref{eq:wave_neu} can be explicitly derived using the \emph{image source method}, a common technique in acoustics for modeling specular reflections, see for instance \cite{Allen1976ImageMF, kuttruff2012room}. It is based on the observation that any specular reflection of an impulsive sound source off a wall can be modeled by introducing a virtual source located outside the domain, which is the symmetric counterpart of the original source with respect to the wall.
Allen and Berkley \cite{Allen1976ImageMF} employed the image-source method to model the full reverberation of an impulse in a rectangular room, facilitating efficient simulation of room impulse responses. Their approach involves iteratively applying the image-source technique, generating virtual sources by reflecting the original source across the walls encountered along a given reflection path. 
As a result, the distributional solution to system \eqref{eq:eq:wave_neu} can be expressed as a series of functions, with each term corresponding to a virtual source.
By adopting a coordinate system aligned with the walls and placing the origin at one of the room's vertices, a straightforward expression for the coordinates of these image sources can be derived, leading to the set of image source points
\begin{equation}\label{eq:IS_coord}
  I_{\Omega} := \{\rvect_{\qvect,\varepsilonvect} = \varepsilonvect \odot \src + 2 \qvect\odot \vvect_L, \quad \varepsilonvect \in \{-1\ts{,}1\}^3,\quad \qvect \in \Z^3\}
\end{equation}
where $\odot$ denotes the Hadamard product\footnote{For two matrices $A$ and $B$ of the same dimension $m \times n$, the Hadamard product $A \odot B$ is a matrix of the same dimension as the operands, with elements given by $(A\odot B)_{ij}=A_{ij}B_{ij}$.}, $\vvect_L=[L_x,L_y,L_z]^\top$ is the room size vector and $\src$ contains the coordinates of the source location. In this setup, $\src$ contains the distances of the source to the walls that define the origin.

Considering all image sources allows us to account for every possible reflection path of the original sound wave on the room's walls, \emph{i.e.} the reverberation. 
In \cite{Allen1976ImageMF}, a formal proof is provided for expressing the solution to \eqref{eq:wave_neu} using the image source method, leading to the following result.
\begin{proposition}[Image source method] \label{prop:IS}
    The solution $p$ to \eqref{eq:wave_neu} in the sense of distributions is given by:
    \begin{equation}
p(\rvect,t) = \sum_{\rvect_{\qvect,\varepsilonvect}\in I_\Omega}  p_{\rvect_{\qvect,\varepsilonvect}}(\rvect,t), \quad \text{with} \quad p_{\rvect_{\qvect,\varepsilonvect}}(\rvect,t)=  \frac{\delta(t-\Ltn{\rvect_{\qvect, \varepsilonvect}-\rvect}/c)}{4\pi \Ltn{\rvect_{\qvect, \varepsilonvect}-\rvect}},\quad (\rvect,t)\in{\R^3\times\R_+}. 
\end{equation}
\end{proposition}
We also refer the reader to \cite[Chapter~4]{sprunckThesis} for a more detailed proof of this result.

Each $p_{\rvect_{\qvect,\varepsilonvect}}$ is a Green kernel for the 3D wave equation in free field\footnote{Note that, $t$ being given, $p_{\rvect_{\qvect,\varepsilonvect}}(\cdot,t)$ is the distribution defined by
$$
\langle p_{\rvect_{\qvect,\varepsilonvect}}(\cdot,t),\varphi\rangle_{\mathcal{S}'(\R^3),\mathcal{S}(\R^3)}=\frac{1}{4\pi ct}\int_{S(\rvect_{\qvect, \varepsilonvect},ct)}\varphi(\rvect)\, d\rvect,
$$
where $S(\rvect_{\qvect, \varepsilonvect},ct)$ denotes the two-dimensional sphere in $\R^3$ centered at $\rvect_{\qvect, \varepsilonvect}$, with radius $ct$.
}. The sum can be interpreted as a superposition of sound waves emitted by a set of point sources located at the positions defined in \eqref{eq:IS_coord}. 
Each image source represents a specific path of specular reflections of the original source on the room's walls and is constructed by iteratively reflecting the source across the encountered walls, see Fig. \ref{fig:ism} for an illustration.

\begin{figure}
    \centering
    \includegraphics[width=0.7\linewidth]{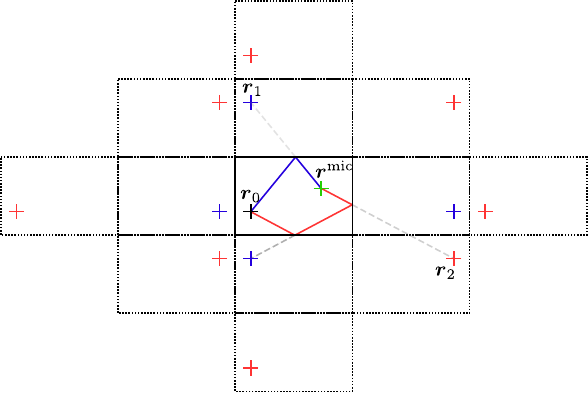}
    \caption{Representation of the source in black, the ﬁrst order image-sources in blue and second order sources in red in 2D. The reﬂection paths from the source $\src$ to a microphone $\mic$ and corresponding to
$\rvect_1$, $\rvect_2$ are drawn in blue and red.}
    \label{fig:ism}
\end{figure}

Notably, the geometric construction of image sources can be generalized to any polyhedral room configuration \cite{borish1984extension} by incorporating additional source visibility constraints. 
However, the image source technique provides the sound field solution to system \eqref{eq:wave_neu} only for a very limited number of room geometries, including cuboid rooms, as stated by Proposition \ref{prop:IS}.

\begin{remark}[An empirical approach to modeling more general absorption conditions]\label{rk:wallAbs}
The image source technique does not account for general admittance conditions such as \eqref{eq:bc_time}. In practice, a heuristic derived from the image source method is often used: a reflection coefficient is assigned to each wall based on its material properties, and amplitude coefficients $a_{\qvect, \varepsilonvect}$ are added to each impulse source term in equation \eqref{prop:IS} to model wall absorption. The amplitude of each source is determined by the product of the reflection coefficients of the walls encountered in the corresponding reflection path, with multiplicity. 
The resulting sound field also solves a free-field equation of the form 
\begin{equation}
    \left\{\label{eq:wave_equation_Free_neu}
    \begin{array}{ll}
     \displaystyle    \frac{1}{c^2}\partial_t^2 p(\rvect,t) - \Delta p(\rvect,t) = \sum_{\substack{\varepsilonvect \in \{-1\ts{,}1\}^3\\ \qvect \in \Z^3}} a_{\qvect, \varepsilonvect} \delta(t)\delta(\rvect-\rvect_{\qvect,\varepsilonvect})& (\rvect,t)\in\R^3 \times\R \\
        p(\rvect,t) = \partial_t p (\rvect,t) = 0 & (\rvect,t)\in\R^3\times \R_-^* ,
    \end{array}\right.  
\end{equation}
with added amplitudes $a_{\qvect, \varepsilonvect}$ for each source term. However, including these amplitudes breaks the direct connection to the original wave equation \eqref{eq:wave_neu}.
Further extensions of this model include considering frequency-dependent coefficients $\alpha_l$ and source directivity, see for instance \cite{srivastava2023realism, di2020echo}. We will only consider the case of constant coefficients in this article.
This
method falls into the category of geometric acoustic methods and provides an accurate approximation of
the pressure field at high frequencies when the wavelength is sufficiently smaller than the dimensions of the
room \cite{kuttruff2012room}. This property makes geometric acoustic methods a good substitute to expensive wave-based
simulators, such as finite elements, at high frequency.
\end{remark}

\subsection{A finite dimensional inverse problem}\label{sec:modelFiniteD}
We aim to recover the positions of the image sources based on the pressure field $p$ measured at a finite number of discrete receivers (microphones) in the room. 

\paragraph{The pressure field $p$ inside the room}
According to the discussion in Section~\ref{sec:directModel}, in an ideal configuration with perfectly reflecting walls\footnote{This is equivalent to imposing Neumann boundary conditions on the pressure field on the boundary of the room.}, the image source method summarized in Proposition~\ref{prop:IS} amounts to saying that the pressure field $p$ solves the following free-field equation:
\begin{equation}\label{eq:wave_freefield}
    \frac{1}{c^2}\frac{\partial^2p}{\partial t^2}(\rvect,t)-\Delta p(\rvect,t) = \psi(\rvect)\delta(t),\qquad (\rvect,t)\in \R^3\times \R_+ 
\end{equation}
where $\psi(\rvect)=\sum_{k=1}^{+\infty} a_k\delta (\rvect-\rvect_k)$, $\rvect_k$ being the locations of the image sources  ranging over the set $I_\Omega$ defined by \eqref{eq:IS_coord} and $a_k=1$ for every $k\in \N-\{0\}$. %
Therefore, the pressure field $p$ reads 
\begin{equation}\label{expr:p}
p(\rvect,t)=\sum_{k=1}^{+\infty}a_k\frac{\delta(\Ltn{\rvect_k-\rvect} -ct)}{4\pi \Ltn{\rvect_k-\rvect} }, \qquad \rvect \in \R^3.
\end{equation}

In accordance with Remark~~\ref{rk:wallAbs}, we will model wall absorption using coefficients $a_k$ that are no longer all equal to 1, but are positive and less than 1, following the model introduced in \cite{Allen1976ImageMF}. Consequently, from now on, the pressure field will still be given by expression~\eqref{expr:p}, but with unknown coefficients $a_k$ matching the reflection properties of the walls.

\paragraph{Observation of the pressure field at each microphone}Let $M\in \N-\{0\}$ be the number of used microphones and 
\begin{equation}\label{def:EM}
E_M=\{\mic_m,\;m\in\dint{1}{M}\}
\end{equation}
be the set of microphone positions.  In order to avoid the singularity of the Green kernel at each microphone location, we assume that for all $k\in \N-\{0\}$, $\rvect_k\notin  \{\mic_m\}_{m\in\dint{1}{M}}$. 
In our model, we need to account for three limitations:
\begin{itemize}
\item Microphones are unable to measure very high frequencies.
\item  Microphones cannot measure continuous signals.
\item The source amplitudes decrease geometrically with the order of reflection, meaning that if $k$ is large, $a_k$ can be considered negligible. Therefore, we will assume that:
$$
\exists K\in \N-\{0\}\quad \mid \quad \forall k\geq K+1, \quad a_k=0.
$$
\end{itemize}
Let us clarify the first and second limitations.
The measured pressure field at each receiver is obtained by convolving $p$ in time with a continuous filter $\kappa$ that models the microphone's response. In our case we consider a low-pass filter that models the limitation of measuring only low-frequency signals. The resulting signal is then discretized into $N$ time steps according to a fixed sampling frequency $f_s$, ranging from $0$ to $T_{\max}=(N-1)/f_s$. Thus, the microphone $m$ provides a sampled version of the signal in the form of a vector $(x_{m,n})_{0\leq n\leq N-1}$ given by 
\begin{equation}\label{eq:discrete_signal}
x_{m,n}=\left(\kappa \ast p(\mic_m,\cdot)\right) (n/f_s) = \sum_{k=1}^Ka_k\frac{\kappa(n/f_s-\Ltn{\rvect_k-\mic_m}/c)}{4\pi\Ltn{\rvect_k-\mic_m}}
\end{equation}
for every $ (m,n)\in\dint{1}{M}\times\dint{0}{N-1}$.

This leads us to define an observation function $\Gamma^K$ mapping a set of $K$ source amplitudes $\avect=(a_k)_{1\leq k\leq K}$ and positions $\rvect=(\rvect_k)_{1\leq k\leq K}$ to an ideal observation vector: 
\begin{equation}\label{eq:gammak}
   \forall (\avect,\rvect) \in (\R_+)^K\times ( \R^3-\{\mic_m\}_{m\in\dint{1}{M}})^K,\qquad
   \Gamma^K(\avect,\rvect) = \sum_{k=1}^K a_k\gamma(\rvect_k),
\end{equation}
where the function $\gamma :\R^3\setminus E_M\rightarrow \R^{MN}$ is defined component-wise by:
\begin{equation}
   \forall (m,n)\in\dint{1}{M}\times\dint{0}{N-1},\ \forall \rvect\in \R^3\setminus E_M,\quad
   \gamma_{m,n}(\rvect) = \frac{\kappa(n/f_s-\Ltn{\rvect-\mic_m}/c)}{4\pi\Ltn{\rvect-\mic_m}}.
\end{equation}
\begin{remark}
The model we consider includes a family of amplification coefficients 
\(\{a_k\}_{1\leq k\leq K}\) belonging to the interval \((0,1)\). Notably, if we have 
a family \(\{a_k\}_{1\leq k\leq K}\) in \(\mathbb{R}_+^*\), it is easy to transform 
it into a family of coefficients within the range \( (0,1) \) by considering the 
normalized family \(\{a_k/A\}_{1\leq k\leq K}\), where $A=\sum_{k=1}^{K}a_k$.
With these new coefficients, the solution \(p\) of the system \eqref{eq:wave_freefield} 
becomes \(p/A\), by linearity. For this reason, we will henceforth assume that 
the amplification coefficients \(\{a_k\}_{1\leq k\leq K}\) belong to \(\mathbb{R}_+^*\).
\end{remark}
Let 
$$
\mathscr{C}=\bigcap_{m=1}^M \overline{B\left(\mic_m, cT_{\max}\right)}\setminus E_M
$$ 
be the set of spike positions that are observable by every microphone in the time interval, \emph{i.e.} the set of sources for which every time of arrival at the microphones is inferior to the final time $T_{\max}$. 
Since the number of Dirac measures (or "spikes") to reconstruct is assumed to be lower than $K$, the reconstruction task can be framed as a least squares optimization problem:%
\begin{equation}
\boxed{\label{eq:nonconv}
    \tag{\mbox{$\mathscr{P}^{K}$}}
    \inf_{(\avect, \rvect)\in \Oset{}{K}}T(\avect,\rvect)\quad \text{with}\quad T(\avect,\rvect)=\frac{1}{2}\Ltn{\xvect-\sum_{k=1}^K a_k\gamma(\rvect_k)}^2 \quad \text{and}\quad \Oset{}{K} =\R_+^K\times\mathscr{C}^K}
\end{equation}
where $\xvect=(x_{mn})_{ (m,n)\in\dint{1}{M}\times\dint{0}{N-1}}$ is the target observation vector.
{
\begin{remark}
The function $T$, used as a criterion in problem \eqref{eq:nonconv}, is quadratic and convex with respect to the variable $\avect$. In contrast, the function $\gamma$ is assumed to be singular at every point of $E_M$. As a result, the function $T$ is smooth, though not convex, in the variable $\rvect$, except at the points of $E_M$. These singularities may lead to pathological behavior in minimizing sequences, particularly if an accumulation point of a minimizing sequence $(\rvect^l)_{l \in \mathbb{N}}$ coincides with a point in $E_M$. This situation is analyzed in Section~\ref{sec:analysis}.
\end{remark}
}
\medskip

In Section~\ref{sec:analysis}, we discuss the well-posedness of this problem. In particular, we will demonstrate that, without further constraints on the problem data, any outcome is possible (existence or non-existence). This will lead us to consider adding an additional constraint to the problem.

\section{\texorpdfstring{Analysis of Problem \eqref{eq:nonconv}}{Analysis of Problem}}\label{sec:analysis}
\subsection{Well-posedness issues}
In this section, we investigate the existence of solutions for problem \eqref{eq:nonconv}.  
We show that, without additional assumptions on the measurements obtained from the microphones, which may include noise, any scenario is possible. In particular, we present two situations: one where problem \eqref{eq:nonconv} has a solution, and another where it does not. The answers provided in this section are partial, as the conclusions are derived within frameworks that are not necessarily physical. Notably, two characteristics of the problem can lead to non-existence: the function $\gamma$ is singular at the points where the microphones are placed, and no regularization term has been added to the least squares function $T$. This will lead us to consider a slightly modified version of problem \eqref{eq:nonconv}.

\paragraph{Choice of the low-pass filter $\kappa$}
In practical applications, we will use an ideal low-pass filter given by:
    \begin{equation} \label{eq:lowpass_filter}
        \kappa^\text{lp} : t \mapsto \operatorname{sinc}(\pi f_s t),
    \end{equation}
     where $f_s$ is both the sampling frequency and the cutoff frequency of the filter. 
     This filter is designed to pass frequencies up to $f_s/2$, as the Fourier transform of $\kappa^\text{lp}$ is a rectangle function of width $\frac{1}{2}$.
Another commonly used filter is the Gaussian one, defined by 
    \begin{equation} \label{eq:lowpass_filterGauss}
\kappa^\sigma:t\mapsto e^{-\frac{t^2}{2\sigma^2}}. 
\end{equation}

\paragraph{An example of non-existence}

The existence of a solution to Problem~\eqref{eq:nonconv} is not guaranteed in general. Indeed, the spikes of a minimizing sequence for Problem~\eqref{eq:nonconv} may converge to microphone positions. In this paragraph, we detail the construction of counterexamples to the existence. 
Let $M$, $N$ be two integers larger than 2 and $\xvect=(x_{mn})_{ (m,n)\in\dint{1}{M}\times\dint{0}{N-1}}$ denote the synthetic observation vector defined by
    \begin{equation}
        \left\{ \begin{array}{ll}
            \forall m\in\dint{2}{M},\;\forall n \in\dint{0}{N-1}& 
            x_{m,n} = 0\\
            \forall n\in\dint{0}{N-1}&
             x_{1,n} = \alpha\kappa(n/f_s)
        \end{array}\right.
    \end{equation}
where $\alpha>0$. Then if $\kappa$ is continuous and $\kappa(0)>0$, the optimal value for Problem \eqref{eq:nonconv} is zero.  

Indeed, let $(\avect^l, \rvect^l)$ denote the sequence defined by 
\begin{equation}
    \left\{ \begin{array}{ll}
        \forall l\in \N^*,& 
        a_1^l = 4\pi\alpha/l, \quad \rvect_1^l=\mic_1+\uvect/l\\
        \forall l\in \N^*,\;k\in\dint{2}{K}& 
        a_k^l = 0, \quad \rvect_k^l=\mic_1+\uvect/l
    \end{array}\right.
\end{equation}
where $\uvect$ is an arbitrary unit vector. For $l$ large enough $\Gamma^K(\avect^l, \rvect^l)$ is well defined and $\Gamma^K(\avect^l,\rvect^l)$ converges to $\xvect$ as $l$ goes to infinity. 
To further simplify this example, assume that $K=1$, \emph{i.e.} we can place only one spike. Assume by contradiction that there exists a solution $(\avect, \rvect)=(a_1, \rvect_1)$ to problem \eqref{eq:nonconv} such that $a_1>0$ and $\rvect_1\neq \mic_1$. Since the optimal value is $0$, $a_1\gamma_{1,n}(\rvect_1) = \alpha \kappa(n/f_s)$ for all $n$. Let $t_1=\frac{\Ltn{\rvect_1-\mic_1}}{c}>0$ the source's time of arrival at $\mic_1$. We get:
\begin{equation}\label{eq:cex1}
    a_1\frac{\kappa(n/f_s-t_1)}{4\pi ct_1} = \alpha \kappa(n/f_s), \quad \forall n\in\dint{0}{N-1}.
\end{equation}
In particular, evaluating this expression at $n=0$ yields $\frac{a_1}{4\pi ct_1}= \frac{\alpha\kappa(0)}{\kappa(t_1)}$, and  \eqref{eq:cex1} can therefore be rewritten as:
\begin{equation}\label{eq:cex2}
    \forall n\in\dint{0}{N-1},\quad \kappa(n/f_s-t_1) = \kappa(n/f_s)\frac{\kappa(t_1)}{\kappa(0)}.
\end{equation}
Relation \eqref{eq:cex2} is not true in general for all values of $n$, depending on the choice of the filter $\kappa$.
For instance if $\kappa = \kappa^\text{lp}$ as defined in \eqref{eq:lowpass_filter}, then \eqref{eq:cex2} yields $\kappa(n/f_s-t_1)=0$ for every $n\in\dint{1}{N-1}$, which is false if $f_s t_1$ is not an integer. 

If $\kappa =\kappa^\sigma$ (Gaussian filter), we get by definition:
\begin{equation}
    \kappa^\sigma(n/f_s-t_1) = \kappa^\sigma(n/f_s)\kappa^\sigma(t_1)e^{\frac{nt_1}{f_s\sigma^2}}.
\end{equation}
As $\kappa^\sigma(0)=1$, \eqref{eq:cex2} leads to $e^{\frac{nt_1}{f_s\sigma^2}}=1$ and $t_1=0$, \emph{i.e.} $\rvect_1=\mic_1$. For both filter types we get a contradiction, thus problem \eqref{eq:nonconv} does not admit a solution in that case.
\paragraph{Existence may arise}

Finally, we provide an existence criterion for Problem~\eqref{eq:nonconv} under the assumption that the operator $\Gamma^K$ is lower bounded in some sense: 
\begin{definition}\label{def:alb}
    $\Gamma^K$ is said to be \textbf{amplitude lower-bounded} if there exists a constant $C>0$ such that:
    \begin{equation}
        \forall (\avect, \rvect) \in \Oset{}{K}, \quad \Ltn{\Gamma^K(\avect, \rvect)}\geq C\sum_{k=1}^Ka_k.
    \end{equation}
\end{definition}

In what follows, we will make the following general assumption on the kernel $\kappa$:
\begin{equation}
\tag{\ensuremath{H_\kappa}}
\label{hyp:filter}
    \begin{array}{l}
        (i) \quad\text{The function }\kappa \text{ is continuous on }\R, \text{ such that }\kappa(0)>0 \\
        (ii) \quad\lim_{| t|\to +\infty}\kappa(t)=0. 
    \end{array}
\end{equation}

Let us now state the main existence result.

\begin{theorem}\label{th:existence1}
    Let us assume that $\kappa$ satisfies \eqref{hyp:filter} and that $\Gamma^K$ is amplitude lower-bounded by a constant $C>0$. We define the constant :
    \begin{equation}   \label{eq:condth1}
    \phi \coloneqq\inf_{t\in\R_+^*}\sum_{n=0}^{N-1}\frac{\kappa(n/f_s)\kappa(n/f_s-t)}{4\pi ct}
    \end{equation}
    and the coefficients
    \begin{equation}
    \mu_m \coloneqq \sum_{n=0}^{N-1}x_{m,n}\kappa(n/f_s),\quad m\in\dint{1}{M}.
    \end{equation}
    Then Problem~\eqref{eq:nonconv} has a solution whenever one of the following conditions is satisfied:
    \begin{enumerate}[(i)]
        \item $\phi<0$ and for all $m\in\dint{1}{M}$, $\mu_m\leq \frac{2}{C}\phi\Ltn{\xvect}$
        \item $\phi\geq 0$ and for all $m\in\dint{1}{M}$, $\mu_m\leq 0$.
    \end{enumerate}
\end{theorem}

The proof of Theorem \ref{th:existence1} is provided in Section~\ref{proof:th:existence1}. 
We first conclude this section by specifying sufficient conditions on the filter $\kappa$ that ensure the assumption ``$\Gamma^K$ is amplitude lower-bounded'' is satisfied.

\paragraph{On the assumption ``$\Gamma^K$ is amplitude lower-bounded''}
The following criterion provides a sufficient condition on the filter $\kappa$ to ensure amplitude lower-boundedness of $\Gamma^K$ with respect to the number of time samples.

\begin{proposition}\label{prop:alb_discrete}
    Let $f_s\in \R_+^*$ , $N\in\N^*$. Let $\kappa$ satisfy \eqref{hyp:filter} and: 
    \begin{equation} \label{eq:alb_discrete}
        \forall \tau \in \left[0, \frac{N-1}{f_s}\right], \qquad \sum_{n=0}^{N-1} \kappa(n/f_s-\tau) > 0
    \end{equation}
then $\Gamma^K$ is amplitude lower-bounded. 
\end{proposition}

In particular we can apply this result to $\kappa^\text{lp}$ defined in \eqref{eq:lowpass_filter}. 

\begin{corollary}\label{cor:prop:alb_discrete}
    Let $\kappa=\kappa^\text{lp}$, then $\Gamma^K$ is amplitude lower-bounded.
\end{corollary}

In practice criterion \eqref{eq:alb_discrete} can be relaxed to a continuous counterpart  which ensures $\Gamma^K$ is asymptotically amplitude lower-bounded as the number of time samples goes to infinity. The following result encompasses the case of the Gaussian filter $\kappa^\sigma$ given by \eqref{eq:lowpass_filterGauss}.

\begin{corollary} \label{prop:alb_crit}
    Let $T_{\max}\in\R_+^*$ and assume that the filter $\kappa$ verifies: 
    \begin{equation} \label{eq:alb_continuous}
        \forall \tau \in [0, T_{\max}], \quad \int_{0}^{T_{\max}}\!\!\!\!\!\!\!\kappa(t-\tau)dt > 0.
    \end{equation}
Then there exists $N'\in\N^*$ such that $\Gamma^K$ is amplitude lower-bounded for all $f_s$, $N$ that verify $N\geq N'$ and $T_{\max}=(N-1)/f_s$.
\end{corollary}

\subsection{Proof of Theorem~\ref{th:existence1}}\label{proof:th:existence1}

We first study the behavior of the spikes of a minimizing sequence for Problem~\eqref{eq:nonconv} (Lemma \ref{prop:behavior}) where $\kappa$ satisfies \eqref{hyp:filter}, and provide an expression of the optimal value (Lemma \ref{prop:limit_cost}). From this expression we deduce a simple existence criterion (Lemma \ref{prop:lemma_crit}) which we then apply to prove Theorem \ref{th:existence1}.


The following lemma explores the asymptotical behavior of the amplitudes and locations of a minimizing sequence $(\avect^l,\rvect^l)$.

\begin{lemma}\label{prop:behavior}
    Consider a minimizing sequence $(\avect^l,\rvect^l)$ for Problem~\eqref{eq:nonconv}. Then, up to a subsequence, the sequence of spike positions $(\rvect^l_k)$ satisfies one of the following properties:
    \begin{enumerate}[(i)]
        \item
        $\exists \rvect_k\in \mathscr C,\;
        \exists a_k\in\R_+, \quad \rvect^l_k\xrightarrow[l\rightarrow+\infty]{}\rvect_k$
         and $a_k^l \xrightarrow[l\rightarrow +\infty]{}a_k$
        \item $\exists m_k\in\dint{1}{M},\;\exists \widetilde{a}_k\in\R_+\quad \rvect^l_k\xrightarrow[l\rightarrow+\infty]{}\mic_{m_k}$,
        $a^l_k\xrightarrow[l\rightarrow+\infty]{}0$ and $\frac{a_k^l}{4\pi\Ltn{\rvect_k^l-\mic_{m_k}}} \xrightarrow[l\rightarrow +\infty]{}\widetilde{a}_k$.
    \end{enumerate}
\end{lemma}
\begin{proof}
    Let $k\in\dint{1}{K}$.
    $(\rvect^l)$ is bounded by definition of $\mathscr C$. The amplitudes $(\avect^l)$ are also bounded as 
    \begin{equation}
        \Ltn{\Gamma^K(\avect^l,\rvect^l)-\xvect}\geq\left\vert\Ltn{\Gamma^K(\avect^l,\rvect^l)}-\Ltn{\xvect}\right\vert\geq C\sum_{k=1}^Ka_k^l-\Ltn{\xvect}
    \end{equation}
    by amplitude lower-boundedness of $\Gamma^K$. We can thus consider a subsequence (still denoted $(\avect^l,\rvect^l)$ with a slight abuse of notation) for which the amplitudes $(\avect^l)$ converge to a certain vector $\avect\in\R_+^K$ and each position $(\rvect_j^l)$ converges to a location $\rvect_j\in\mathscr C\cup E_M$. Case $(i)$ is verified if $\rvect_k\notin E_M$.

    Consider now a spike sequence $(\rvect_k^l)$ converging to a sensor position $\mic_{m_k}$. We define $I_{m_k}\subset\dint{1}{K}$ as the indices of the spikes locations $\rvect_j^l$ converging towards $\mic_{m_k}$ as $l\rightarrow +\infty$. The residual at the first time sample is given by
    \begin{equation}
    c^l_{m_k,0}=x_{m_k,0}-\sum_{j\notin I_{m_k}} a_j^l\gamma_{m_k,0}(\rvect_j^l)-\sum_{j\in I_{m_k}} a_j^l\gamma_{m_k,0}(\rvect_j^l)
    \end{equation}
    
By minimality, $(c^l_{m_k,0})^2$ is necessarily bounded, thus so is $\sum_{j\in I_{m_k}}a_j^l\gamma_{m_k,0}(\rvect_{j}^l)$. Each term of the sum being positive, they are bounded, from which we deduce that  
$$
a_k^l\gamma_{m_k,0}(\rvect_k^l)=a_k^l\frac{\kappa(\Ltn{\rvect_k^l-{\mic_{m_k}}}/c)}{4\pi\Ltn{\rvect_k^l-\mic_{m_k}}}
$$ 
is bounded. By continuity of $\kappa$ at $0$, $\frac{a_k^l}{\Ltn{\rvect_k^l-\mic_{m_k}}}$ is bounded, and therefore the case $(ii)$ arises.
\end{proof}

Using Lemma \ref{prop:behavior}, we now provide a useful expression of the optimal value.

\begin{lemma}\label{prop:limit_cost}
    There exists an integer (possibly equal to 0) $K'\leq K$, a pair $(\avect, \rvect)\in(\R_+)^{K'}\times (\R^3\setminus E_M)^{K'}$, and $\widetilde{\avect}\in\R_+^M$ such that, up to a permutation of the indices, the optimal value of problem \eqref{eq:nonconv} expands as:
    \begin{equation} \label{eq:limitvalue}
        \inf_{(\avect, \rvect)\in\Oset{}{K}}T(\avect, \rvect)=\widetilde{T}(\avect, \rvect,\widetilde{\avect}) \coloneqq \frac{1}{2}\sum_{m=1}^M\sum_{n=0}^{N-1}\left(x_{m,n}-\sum_{k=1}^{K'}a_k\gamma_{m,n}(\rvect_k)-\widetilde{a}_m\kappa(n/f_s)\right)^2
    \end{equation}

    Moreover, there exists a minimizing sequence $(\avect^l, \rvect^l)$ for Problem \eqref{eq:nonconv} such that $(a_k^l,\rvect_k^l)$ converges to $(a_k,\rvect_k)$ for all $k\in\dint{1}{K'}$ and $(a_k^l)$ converges to $0$ for $k\in\dint{K'+1}{K}$.
\end{lemma} 

\begin{proof}

Consider an arbitrary minimizing sequence $(\avect^l,\rvect^l)$ for problem \eqref{eq:nonconv}.
Lemma \ref{prop:behavior} shows that, up to extracting a sub-sequence and permuting the indices, we can assume that the first $K'\in\dint{0}{K}$ spikes positions converge to locations $\rvect_k$ that are distinct from the sensor locations, while the remaining spikes converge to positions in $E_M$.
By continuity of the kernel, $\sum_{k=1}^{K'}a_k^l\gamma_{m,n}(\rvect_k^l)$ converges to $\sum_{k=1}^{K'}a_k\gamma_{m,n}(\rvect_k)$ for all $m, n$.

Consider now the spikes that converge to some microphone location  in $E_M$. 
For $m\in\dint{1}{M}$ we define as in the previous proof $I_m \subset \dint{K'+1}{K}$ as the set of indices $k$ such that $\rvect_k=\mic_m$. 
Observe that by Lemma \ref{prop:behavior} if several spikes converge to the same microphone $m$ their contributions share the same sign and can then be summed: 
\begin{equation}
    \forall m, n ,\quad \sum_{k\in I_m} \ts{a_{k}^l\gamma_{m,n}(\rvect^l_{k})} \xrightarrow[l\rightarrow+\infty]{}
    \kappa(n/f_s)\sum_{k\in I_m} \widetilde{a}_{k}=\kappa(n/f_s)\widetilde{a}_m,\quad \widetilde{a}_m\in\R_+.
\end{equation}
Moreover, if a spike $\rvect^l_{k}$ converges to a microphone location $\mic_m$, the corresponding amplitude $a_k^l$ converges to $0$, hence:
\begin{equation}
    \forall m\neq m'\in\dint{1}{M}, \; \forall k \in I_{m'},\quad a_{k}^l\gamma_{m,n}(\rvect^l_{k}) \xrightarrow[l\rightarrow+\infty]{} 0.
\end{equation}
Thus, a spike that converges to a microphone contributes only to the terms related to that particular receiver in the cost function, which justifies formula \eqref{eq:limitvalue}. If none of the spikes converge to a given microphone $m$, the corresponding coefficient $\widetilde{a}_m$ is zero. 
\end{proof}

In the following lemma, we state a numerical condition that guarantees the existence of a minimizer for Problem~\eqref{eq:nonconv}.
\begin{lemma}\label{prop:lemma_crit}
    Let $K'\leq K$, $(\avect, \rvect)\in(\R_+)^{K'}\times (\R^3\setminus E_M)^{K'},\;\widetilde{\avect}\in\R_+^M$
    such that $\inf_{(\avect, \rvect)\in\Oset{}{K}}T(\avect, \rvect)=\widetilde{T}(\avect, \rvect,\widetilde{\avect})$. 
     If the pair $(\avect, \rvect)$ is such that
     \begin{equation}     \label{eq:ineq_criterion}
         \forall m\in\dint{1}{M},\quad \sum_{n=0}^{N-1}x_{m,n}\kappa(n/f_s)\leq \sum_{k=1}^{K'}\sum_{n=0}^{N-1}a_k\kappa(n/f_s)\gamma_{m,n}(\rvect_k),
     \end{equation}
    then $\widetilde \avect = 0$ and Problem~\eqref{eq:nonconv} has a solution.
\end{lemma}

\begin{proof}

Let $K'<K$, $(\avect, \rvect,\widetilde \avect)$ yielding a decomposition of the optimal value as specified in Lemma \ref{prop:limit_cost} and consider the following quadratic optimization program
\begin{equation}\label{pb:limit_cost}
    \inf_{\bvect\in\R^M}\widetilde{T}(\avect, \rvect,\bvect)
    =\inf_{\bvect\in\R^M}\sum_{m=1}^M \widetilde{T}^m(b_m)
\end{equation}
where $\widetilde{T}^m:t\mapsto \frac{1}{2} \sum_{n=0}^{N-1}\left(x_{m,n}-\sum_{k=1}^{K'}a_k\gamma_{m,n}(\rvect_k)-t\kappa(n/f_s)\right)^2$. Note that $\widetilde{T}^m$  is a positive, convex quadratic polynomial of the real variable $t$. Denoting by $\widetilde{b}_m^*$ the minimizer of $\widetilde{T}^m$ over $\R$, its minimizer over $\R_+$ is necessarily $\max(\widetilde{b}_k^*,0)$. 
We deduce that if every component $\widetilde{b}_m^*$ is negative, then the coefficients $\widetilde{a}_m$ are all zero and consequently there exists a solution to problem \eqref{eq:nonconv}.
Indeed, we obtain in this case $T(\avect',\rvect') = \inf_{(\avect, \rvect)\in\Oset{}{K}}T(\avect, \rvect)$, with 
\begin{equation}
    (a_k',\rvect_k')= \left\{\begin{array}{cc} \label{eq:fun_orthogonal}
        (a_k,\rvect_k) & \text{if } k\leq K'\\
        (0, \vvect) & \text{otherwise}
    \end{array}
    \right.
\end{equation}
where $\vvect$ is an arbitrary location distinct from the microphones.
Note that $\widetilde{\bvect}^*$ is given by the first order optimality conditions for each function $\widetilde{T}^m$:
\begin{equation} \label{eq:opticond}
   \forall m\in\dint{1}{M},\quad
   \sum_{n=0}^{N-1}\kappa(n/f_s)\left(x_{m,n}-\sum_{k=1}^{K'}a_k\gamma_{m,n}(\rvect_k)-\widetilde{b}^*_m\kappa(n/f_s)\right)=0
\end{equation}
and thus,
\begin{equation}
   \forall m\in\dint{1}{M},\quad
   \sum_{n=0}^{N-1}\kappa(n/f_s)^2\widetilde{b}^*_m=\sum_{n=0}^{N-1}\kappa(n/f_s)\left(x_{m,n}-\sum_{k=1}^{K'}a_k\gamma_{m,n}(\rvect_k)\right).
\end{equation}

Since $\sum_{n=0}^{N-1}\kappa(n/f_s)^2\geq \kappa(0)^2 >0$, we infer:
\begin{equation}
    \forall m\in\dint{1}{M},\quad\widetilde{b}^*_m\leq 0 \iff \sum_{n=0}^{N-1}\kappa(n/f_s)\left(x_{m,n}-\sum_{k=1}^{K'}a_k\gamma_{m,n}(\rvect_k)\right)\leq 0
\end{equation}
 which is exactly \eqref{eq:ineq_criterion}.
\end{proof}

We can now prove Theorem \ref{th:existence1}.

\begin{proof}[Proof of Theorem \ref{th:existence1}]
In order to get a global existence criterion on the observation vector $\xvect$ and the operator $\Gamma^K$, we only need to compute a uniform lower bound of the right hand side of inequality \eqref{eq:ineq_criterion}. We consider a decomposition of the optimal value as given by Lemma~\ref{prop:limit_cost} and keep the same notations.

Consider $\phi$ as defined in Theorem \ref{th:existence1}. Using \eqref{hyp:filter}, we infer that $\phi$ is finite by continuity and boundedness of the kernel $\kappa$. 
Let $m\in\dint{1}{M}$. We distinguish two cases based on the sign of $\phi$. \\

\noindent {\it Proof under the assumption $(i)$.} Assume $\phi\leq 0$. Consider $(\avect, \rvect,\widetilde{\avect})$ a decomposition of the optimal value as given by Lemma \ref{prop:limit_cost}, $K'$ the associated truncating integer, and a corresponding minimizing sequence $(\avect^l,\rvect^l)$. If taking a null vector of amplitudes yields a solution to problem \eqref{eq:nonconv}, then we are done. Otherwise, using the amplitude lower-boundedness hypothesis we obtain the following inequalities for all $l$ large enough:
$$
   \Ltn{\xvect-\Gamma^K(0,\rvect^l)}=\Ltn{\xvect} \geq \Ltn{\xvect-\Gamma^K(\avect^l,\rvect^l)}\geq
   \Ltn{\Gamma^K(\avect^l,\rvect^l)} - \Ltn{\xvect} \geq C \sum_{k=1}^{K}a_k^l - \Ltn{\xvect}.
$$
Letting $l$ go to infinity, we obtain $\frac{2}{C}\Ltn{\xvect}\geq \sum_{k=1}^{K'}a_k$, from which we infer
$$
     \quad\sum_{n=0}^{N-1}\sum_{k=1}^{K'}\kappa(n/f_s)a_k\gamma_{m,n}(\rvect_k)=\sum_{k=1}^{K'}a_k\sum_{n=0}^{N-1}\frac{\kappa(n/f_s)\kappa(n/f_s-\Ltn{\rvect_k-\mic_m}/c)}{4\pi\Ltn{\rvect_k-\mic_m}}\geq \frac{2\phi}{C}\Ltn{\xvect}.
$$
By using $(i)$, we finally get that  \eqref{eq:ineq_criterion} is true, whence the result.\\

\noindent {\it Proof under the assumption $(ii)$.} Assume $\phi\geq 0$. Then we obviously have:
\begin{equation}
 \forall m\in\dint{1}{M},\quad   \sum_{n=0}^{N-1}\sum_{k=1}^{K'}\kappa(n/f_s)a_k\gamma_{m,n}(\rvect_k)
     \geq 0\geq \mu_m,
\end{equation}
and \eqref{eq:ineq_criterion} is true.
Applying Lemma \ref{prop:lemma_crit} yields the expected conclusion.
\end{proof}

\subsection{Proofs of Proposition~\ref{prop:alb_discrete}, Corollary~\ref{cor:prop:alb_discrete} and Corollary~\ref{prop:alb_crit}}

\begin{proof}[Proof of Proposition~\ref{prop:alb_discrete}]  
 We only need to consider amplitude vectors $\avect$ such that $\sum_k a_k > 0$. By equivalence of the norms in finite dimension and homogeneity, considering the set of convex weights $H=\{\alphavect\in\R_+^K,\;\sum_{k=1}^K\alpha_k=1\}$, one has to show:
     \begin{equation}
        \exists C\in\R^*_+,\;\forall r\in \mathscr{C}^{K},\; \forall \alpha\in H, \quad \Lon{\Gamma^K(\alpha,r)}\geq C .
    \end{equation}
 Let \ts{$J:\avect, \rvect\mapsto \Lon{\Gamma^K(\avect, \rvect)}$} and $(\alphavect^l, \rvect^l)$ a minimizing sequence for $\inf_\Lambda J$, where $\Lambda = H\times \mathscr{C}^K$. As $\Lambda$ and $\mathscr{C}$ are bounded, $(\alphavect^l, \rvect^l)$  converges up to a subsequence to some $(\alphavect^*,\rvect^*)$ where $\alphavect^*\in H$ and $\rvect^*\in 
    \left( \mathscr C \cup E_M \right)^K$. 
    Observe that if a spike $\rvect_k$ converges to a microphone location, the corresponding amplitude converges to $0$.
    Indeed, let $m\in\dint{1}{M}$ and let $I_m\subset \dint{1}{K}$ be the set of indices $k$ such that $\rvect_k^*=\mic_m$. Assume that $I_m$ is non-empty, \emph{i.e.} there exists a spike position converging to microphone $\mic_m$.
     We have:
     \begin{equation}
        \Lon{\Gamma^K(\avect^l,\rvect^l)} \geq \sum_{k=1}^K\alpha_k^l\gamma_{m,0}(\rvect_k^l)\underset{l\rightarrow +\infty}{\sim} \sum_{k\in I_m}\alpha_k^l\gamma_{m,0}(\rvect_k^l) + \sum_{k\notin I_m}\alpha_k^*\gamma_{m,0}(\rvect_k^*).
     \end{equation}
     The sum $\sum_{k\in I_m}\alpha_k^l\gamma_{m,0}(\rvect_k^l)$ being bounded and since $\kappa(0)>0$, each term of the sum is positive and must consequently be bounded:
    \begin{equation}
        \alpha_k^l\gamma_{m,0}(\rvect_k^l)\underset{l\rightarrow +\infty}{\sim}\frac{\alpha^l_k}{4\pi\Ltn{\rvect_k^l-\mic_m}}\kappa(0) =\operatorname{O}\left(1\right), \quad \forall k\in I_m.
    \end{equation}
    
    Hence if $(\rvect_k^l)$ converges to a microphone $m$ for some $k$, then $\alpha_k^*=0$ and $\frac{\alpha_k^l}{4\pi\Ltn{\rvect_k^l-\mic_m}}$ converges up to a subsequence to some nonnegative value $c^*_k$. Let $0<K'\leq K''\leq M$ such that:
    \begin{equation}
        \left\{\begin{array}{l}
            \forall k\in\dint{1}{K'},\; \alpha_k^*>0 \text{ and } \rvect_k^*\in\mathscr{C}\\
            \forall k\in\dint{K'+1}{K''},\; \alpha_k^*=0\text{ and } \lim_{l\rightarrow +\infty}\rvect_k^l=\mic_{m_k}\in E_M\\
            \forall k\in\dint{K''+1}{K},\; \alpha_k^*= 0 \text{ and } \rvect_k^*\in\mathscr{C}.
        \end{array}\right.
    \end{equation}
    Note that $K'\neq 0$ as $\alpha^*\in H$.
    We have:
    \begin{equation}
\forall l\in\N, \quad    J (\avect^l, \rvect^l) = 
       \sum_{m=1}^M\sum_{n=0}^{N-1}\left\vert\sum_{k=1}^{K}\alpha_k^l\gamma_{m,n}(\rvect^l_k)\right\vert\geq 
        \sum_{m=1}^M\sum_{n=0}^{N-1}\sum_{k=1}^{K}\alpha_k^l\gamma_{m,n}(\rvect^l_k)
        .
    \end{equation}
    Letting $l$ go to infinity, we get:

    \begin{equation}\label{eq:ineqJ}
        \inf_\Lambda J \geq \sum_{m=1}^{M}\sum_{k=1}^{K'}\frac{\alpha_k^*}{4\pi\Ltn{\rvect_k^*-\mic_m}}\sum_{n=0}^{N-1}\kappa\left(\frac{n}{f_s}-\frac{4\pi\Ltn{\rvect^*_k-\mic_m}}{c}\right)+\sum_{k=K'+1}^{K''}c_k^*\sum_{n=0}^{N-1}\kappa(n/f_s).
    \end{equation}
    By construction of $\mathscr{C}$, $0\leq \frac{\Ltn{\rvect_k^*-\mic_m}}{c}\leq(N-1)/f_s$ for $1\leq k\leq K'$, $1\leq m\leq M$. 
    Inequality \eqref{eq:alb_discrete} then guarantees that the term $\frac{\alpha_k^*}{4\pi\Ltn{\rvect_k^*-\mic_m}}\sum_{n=0}^{N-1}\kappa\left(\frac{n}{f_s}-\frac{\Ltn{\rvect_k^*-\mic_m}}{c}\right)$  is positive for all $m\in\dint{1}{K'}$ and $k\in\dint{1}{K'}$. Likewise, the terms $c_k^*\sum_{n=0}^{N-1}\kappa\left(\frac{n}{f_s}\right)$ are nonnegative for $K'<k\leq K''$, thus the right-hand side in \eqref{eq:ineqJ} is positive. 
\end{proof}

\begin{proof}[Proof of Corollary~\ref{cor:prop:alb_discrete}] 
   We call $\kappa^\text{lp}$ the continuous extension of $\kappa^\text{lp}$ at $0$. 
    We will prove that $\kappa^\text{lp}$ satisfies \eqref{eq:alb_discrete}.
    Let $\tau\in[0,(N-1)/f_s]$, $n^+\in \N$ the smallest index such that $n^+/f_s\geq\tau$, and $n^-$ the largest index such that $n^-/f_s\leq\tau$. Let us set $y_n=\operatorname{sinc}(\pi f_s (n/f_s-\tau))$ for all $n$. One expands:
    \begin{equation}
        \sum_{n=0}^{N-1} y_n=
        \sum_{n=0}^{n^-} y_n
        +\sum_{n=n^+}^{N-1} y_n - \delta_{n^-,n^+},
    \end{equation}
where $\delta$ denotes the Kronecker symbol and handles the case where $\tau$ is exactly equal to one of the time samples, as $\kappa^\text{lp}(0) = 1$.
    \begin{figure}[h!] 
    \centering 
    \includegraphics[width=0.3\linewidth]{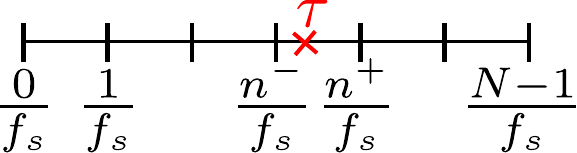}
    \caption{} \label{fig:notations_proof_lab}
\end{figure}  

Let us first assume that $n^-<n^+$. 
Let $l\in\dint{0}{N-n^+-1}$, we have
    \begin{equation}
     y_{n^+ +l}=
    \frac{\sin(\pi f_s(n^+/f_s-\tau)+l\pi)}{\pi f_s((n^+ +l)/f_s-\tau)}=
    (-1)^l\frac{\sin(\pi f_s(n^+/f_s-\tau))}{\pi f_s((n^+ +l)/f_s-\tau)}.
    \end{equation}
    Note that $\sin(\pi f_s(n^+/f_s-\tau))>0$ as $\pi f_s(n^+/f_s-\tau)\in(0,\pi)$, and $y_{n^+}>0$.
    By the alternating series theorem, the sum carries the sign of its first term. Indeed, we get the upper bound:
        \begin{equation}
        \left\vert\sum_{n=n^+ +1}^{N-1}y_n\right\vert\leq \vert y_{n^++1}\vert=
        \frac{\sin(\pi f_s(n^+f_s-\tau))}{\pi f_s((n^+ +1)/f_s-\tau)}
        \end{equation}
    thus $\sum_{n=n^++1 }^{N-1}y_n\geq-\vert y_{n^++1}\vert$ and
$$
\sum_{n=n^+ }^{N-1}y_n\geq y_{n^+}-\vert y_{n^++1}\vert =
    \sin(\pi f_s(n^+f_s-\tau))\left(\frac{1}{\pi f_s(n^+/f_s-\tau)}-\frac{1}{\pi f_s(n^+ +1)/f_s-\tau)}\right)>0.
$$
Likewise, the sum $\sum_{n=0}^{n^-} y_n$ is non-zero and shares the sign of $\frac{\sin(\pi f_s(n^-/f_s-\tau))}{\pi f_s(\tau-n^-/f_s)}$, which is also positive.
        
    In the case where $n^-=n^+$, we have $\tau=n^+/f_s$ and $y_{n^+}=1$ and then,
    \begin{equation}
       \forall n\in \dint{0}{N-1},\quad \pi f_s(n/f_s-\tau)=\pi (n-n^+)\in\pi\mathbb{Z}.
    \end{equation}
It follows that $y_n=\delta_{n,n^+}$ and $\sum_{n=0}^{N-1} y_n=1$, which concludes the proof.
\end{proof}

\begin{proof}[Proof of Corollary~\ref{prop:alb_crit}]
 Let $\tau \in [0,T_{\max}]$ be given. 
 Assume that the sampling frequency $f_N$ is chosen such that $T_{\max} = (N-1)/f_N$. 
  Observing that the sum in inequality \eqref{eq:alb_discrete} is a Riemann sum
    \begin{equation}
        S_N(\tau)\coloneqq\sum_{n=0}^{N-1} \kappa(n/f_N-\tau) = \sum_{n=0}^{N-1} \kappa\left(n\frac{T_{\max}}{N-1}-\tau\right),
    \end{equation}
it follows that $S_N$ converges uniformly towards $\tau \mapsto \int_{0}^{T_{\max}}\kappa(t-\tau)dt$ over $[0,T_{\max}]$ as $N\to +\infty$.
%
    Thus for $N$ large enough, inequality \eqref{eq:alb_discrete} is satisfied and $\Gamma^K$ is amplitude lower-bounded.
\end{proof}

\section{A super-resolution type algorithm}
\label{sect:IS_inv_theory}

The discussion in Section~\ref{sec:analysis} leads us to define a slightly modified version of Problem~\eqref{eq:nonconv}  to avoid the potential pathologies described above.
In order to avoid non-existence caused by the singularities, we will enforce a small distance $\varepsilon>0$ to the microphone positions. Let us hence introduce
$$
\Reps = \R^3\setminus \bigcup_{m\in\dint{1}{M}}B(\mic_m,\varepsilon).
$$
In our numerical approach, we will reformulate the problem as a BLASSO-type problem. Thus, it is relevant to add a $l_1$ regularization term to the cost function.
We thus consider the problem
\begin{equation}
    \label{eq:nonconv_eps_l1}
    \tag{\mbox{$\ts{\mathscr{P}^{K}_{\lambda, \varepsilon}}$}}
    \boxed{
    \inf_{(\avect, \rvect)\in \Oset{\varepsilon}{K}}T_\lambda(\avect, \rvect) \quad \text{with}\quad \Oset{\varepsilon}{K}=\R_+^{K}\times {(\Reps)}^K \text{ and } 
    T_\lambda(\avect, \rvect) = T(\avect, \rvect)+\lambda\sum_{k=1}^K a_k.}
\end{equation} 

{
\begin{remark}[Numerical perspectives]
In this work, noting that singularities in the criterion could obstruct existence and induce pathological behavior, we chose to modify the space of admissible $\rvect$ by enforcing a minimal distance from the singularities. From a numerical standpoint, this entails an additional projection step, 
which is not activated in most practical cases.
This aspect will be further discussed in Section~\ref{sect:numAlgo}. Nevertheless, alternative strategies based on modifying the criterion itself would have been equally legitimate.
\end{remark}}

\subsection{\texorpdfstring{Analysis of Problem \eqref{eq:nonconv_eps_l1}}{Analysis of Problem}}
The maximal distance constraint on the spike locations is no longer needed, as the regularization term forces an upper bound on the amplitudes $\avect$. Any spike vanishing at infinity hence has a null contribution to the objective function. {We now turn to the question of existence for problem~\eqref{eq:nonconv_eps_l1}, which is ensured under very weak assumptions.} Note that we can also define problem \eqref{eq:nonconv_eps_l1} in the case $\varepsilon=0$ by optimizing the cost function on $\Oset{}{K}$.
\begin{proposition}
    Let $\varepsilon>0$, $\lambda\geq 0$. Then, if at least one of the following assumptions holds, problem \eqref{eq:nonconv_eps_l1} has a solution:
    \begin{enumerate}[(i)]
        \item $\lambda > 0$
        \item $\Gamma^K$ is amplitude lower-bounded (see~Definition~\ref{def:alb}). 
    \end{enumerate}
\end{proposition}
\begin{proof}
    Let $(\avect^l,\rvect^l)$ be a minimizing sequence for problem \eqref{eq:nonconv_eps_l1}.  If $\lambda >0$, $(\avect^l)$ is bounded due to the coercivity of the regularization term. In the case $\lambda =0$, $\Gamma^K$ is amplitude lower-bounded, which implies that $(\avect^l)$ is still bounded. As the amplitudes are bounded and $\Gamma^K$ is continuous, if a given location $\rvect_k$ diverges to infinity, its contribution $a_k \gamma (\rvect_k)$ vanishes in the limit. We can thus replace each diverging spike location of the sequence by an arbitrary location $\vvect$ distinct from the microphones to obtain a bounded minimizing sequence. Any closure point of this sequence is a solution to Problem~\eqref{eq:nonconv_eps_l1}
\end{proof}

%
Notably, the conclusion of Theorem \ref{th:existence1} still holds when considering Problem \eqref{eq:nonconv_eps_l1} with \(\lambda > 0\) and \(\varepsilon = 0\).
\begin{theorem}\label{th:existence2}
Let $\lambda>0$ and let us assume that $\kappa$ is continuous, bounded and $\kappa(0)>0$. Let $\phi$, $\mu_m$ be defined as in Theorem \ref{th:existence1}.
Then if one of the following conditions is true, Problem~\eqref{eq:nonconv_eps_l1} with $\varepsilon=0$ has at least one solution:
\begin{enumerate}[(i)]
    \item $\phi<0$ and for all $m\in\dint{1}{M}$, $\mu_m\leq \frac{\phi}{2\lambda}\Ltn{\xvect}^2$
    \item $\phi\geq 0$ and for all $m\in\dint{1}{M}$, $\mu_m\leq 0$
\end{enumerate}
\end{theorem}

\begin{remark}
    If $\Gamma^K$ is amplitude lower-bounded with constant $C$, the inequality constraint in case $(i)$ can be improved to:
    \begin{equation}
        \forall m\in\dint{1}{M}, \quad \mu_m\leq \max\left(\frac{\phi}{2\lambda}\Ltn{\xvect}^2, \frac{2\phi}{C}\Ltn{\xvect}\right).
    \end{equation}
\end{remark}

\begin{proof}
    We adapt the proof of Theorem \ref{th:existence1}.
    Observe that because $T_\lambda(\avect, \rvect)\geq \lambda \sum_k a_k$ the amplitudes of a minimizing sequence $(\avect^l, \rvect^l)$ are bounded, removing the need for the amplitude lower-boundedness hypothesis in Lemma \ref{prop:behavior}. Lemma \ref{prop:behavior} can then be reproduced, with the added possibility of a spike diverging to infinity. Due to the boundedness of the amplitudes, the contribution of such a spike to the cost function vanishes in the limit. The rest of the proof is identical.

    Likewise, the proof of Lemma \ref{prop:limit_cost} is identical in the case $\lambda>0$, with a different expression of the optimal value:
    \begin{equation}
        \inf_{(\avect, \rvect)\in\Oset{}{K}}T_\lambda(\avect, \rvect)=\widetilde{T}_\lambda(\avect, \rvect,\widetilde{\avect}) \coloneqq \widetilde{T}(\avect, \rvect,\widetilde{\avect})
+\lambda\sum_{k=1}^{K'}a_k.
    \end{equation}
 
    The addition of the regularization term does not affect the argument in Lemma \ref{prop:lemma_crit}, and the proof is identical as the optimality conditions considered in \eqref{eq:opticond} are unchanged. We thus obtain the same existence criterion.

    Finally, we only need to adapt the upper bound on the amplitudes given in the proof of Theorem \ref{th:existence1} to handle the case $\phi < 0$. Using the same notations, we have here $T_\lambda(0,\rvect)=\frac{1}{2}\Ltn{\xvect}^2\geq T_\lambda(\avect, \rvect)\geq \lambda \sum_{k=1}^{K'} a_k$. The same argument as before yields the existence criterion.
\end{proof}

\subsection{A convex relaxation}\label{sect:IPdimInf}
By using the integral representation \eqref{eq:discrete_signal} we extend the definition of the observation function $\Gamma^K$ to a linear operator $\Gamma^\varepsilon$ on the space of Radon measures. The vector space of Radon measures over $\Reps$, denoted by $\mathcal M(\Reps)$, can be defined as the dual space of the space of continuous functions $\mathcal C_0(\Reps)$ that vanish at infinity. $\mathcal M(\Reps)$ is a Banach space when endowed with the total variation norm defined by :
\begin{equation}   
\forall \psi \in \mathcal M(\Reps),\quad \tvn{\psi} = \sup \left\{ \int_X{f d\psi},\; f \in \mathcal C_0(\Reps),\; \infn{f}\leq 1 \right\}.
\end{equation}
Proceeding as in Section~\ref{sec:modelFiniteD}, the solution $ p(\cdot, t)$ to the free-field wave equation \eqref{eq:wave_freefield}, with a spatial source term $\psi$ supported on \( \Reps \), is obtained by convolving \( \psi \) in space with the ``Green's function'' distribution:
\[
G_t = \frac{\delta(t - \|\rvect\|/c)}{4\pi \|\rvect\|}
\]
Then, \( p \) is further convolved in time with the filter \( \kappa \) and discretized at each sample to obtain the observation vector:
\begin{equation}\label{eq:discrete_signal2}
    x_{m,n}=\left(\kappa \ast p(\mic_m,\cdot)\right) (n/f_s) = \int_{\rvect\in\Reps}\frac{\kappa(n/{f_s}-\Ltn{\mic_m-\rvect}/c)}{4\pi\Ltn{\mic_m-\rvect}}\psi(\rvect)d\rvect,\quad  m\in\dint{1}{M},\; n\in\dint{0}{N-1}.
\end{equation}
This leads to introduce the linear operator $\Gamma^\varepsilon$:
\begin{equation}\label{eq:def_gamma}
\Gamma^\varepsilon : 
\begin{array}[t]{rcl}
\M(\Reps) & \longrightarrow & \R^{MN}\\
\psi &\longmapsto&\displaystyle \left(\int_{\rvect \in \Reps} \frac{ \kappa(n/f_s-\Ltn{\rvect-\mic_{m}}/c)}{4\pi \Ltn{\rvect-\mic_{m}}}\, d\psi(\rvect)\right)_{\substack{1\leq m\leq  M\\0\leq n\leq N-1}}
\end{array}
\end{equation}
which can be interpreted as the mapping of a given source term $\psi$ supported on $\Reps$ to the measured free-field response at each receiver location. 
Note that \eqref{eq:discrete_signal2} is well-defined whenever \( \psi \) is a Radon measure supported on \( \Reps \). Furthermore, \( \Gamma^\varepsilon \psi \) extends \eqref{eq:discrete_signal2} to a broader class of \( \psi \), specifically Radon measures supported on \( \Reps \).

The convex relaxation of problem \eqref{eq:nonconv_eps_l1} is called Beurling-LASSO or BLASSO \cite{duval2015exact} and can be written as:
\begin{equation}\label{eq:relaxed_blasso}
\tag{\mbox{$\mathscr{B}_{\lambda,\varepsilon}$}}
\boxed{
    \inf_{\psi\in\mathcal M(\Reps)}\frac{1}{2}\Ltn{\xvect-\Gamma^\varepsilon \psi}^2 + \lambda\tvn{\psi}.}
\end{equation}
\begin{remark}
    Problem \eqref{eq:nonconv_eps_l1} is indeed the restriction of problem \eqref{eq:relaxed_blasso} to linear combinations of $K$ Dirac measures, as for $(\avect, \rvect)\in \Oset{\varepsilon}{K}$, $\Gamma^K(\avect, \rvect) = \Gamma^\varepsilon(\sum_{k=1}^K a_k\delta_{\rvect_k})$ and $\tvn{\sum_{k=1}^K a_k\delta_{\rvect_k}}=\sum_{k=1}^K|a_k|=\Lon{\avect}$. Moreover, $\Gamma^\varepsilon$ is continuous, and for $\lambda > 0$ problem \eqref{eq:relaxed_blasso} admits solutions \cite{bredies2013inverse}, with at least one $MN$-sparse solution \cite{boyer2019representer} (meaning a measure composed of at most $MN$ Dirac masses). In particular, for $K\geq MN$, the optimal values for \eqref{eq:nonconv_eps_l1} and \eqref{eq:relaxed_blasso} are the same. 
\end{remark}

In \cite{duval2015exact}, Duval and Peyré introduced a set of useful tools to analyze the structure of solutions to the BLASSO problem \eqref{eq:relaxed_blasso} in the case where \(\xvect = \Gamma^\varepsilon \psi^* + \evect\), with \(\evect\) being a noise vector and \(\psi^* = \sum_{k=1}^{K} a^*_k \gamma (\rvect^*_k)\).  

In particular, for \(\rvect \in \Reps\), let \(\Gamma^\varepsilon_{\rvect}: \R^{4K} \to \R^{MN}\) be defined in matrix form as follows:
\begin{equation}
  \Gamma^\varepsilon_{\rvect} \coloneqq 
  \begin{bmatrix}
      \gamma(\rvect_1)  \;\ldots\;   \gamma(\rvect_K) & \partial_x\gamma(\rvect_1)  \;\ldots\;   \partial_x\gamma(\rvect_K )&
      \partial_y\gamma(\rvect_1)  \;\ldots\;   \partial_y\gamma(\rvect_K)&\partial_z\gamma(\rvect_1)  \;\ldots\;   \partial_z\gamma(\rvect_K )
  \end{bmatrix}
\end{equation}

In what follows, we denote by $M^\top$ (resp. $M^*$) the transpose of a matrix $M$ (resp. the adjoint of the operator $M$). Assuming $\Gamma^\varepsilon_{\rvect^*}$ has full rank, the so-called \emph{vanishing derivatives precertificate} $\eta_V$ is defined by: 
\begin{equation}\label{def:etaV}
\eta_V=(\Gamma^{\varepsilon})^* v^*,\quad \text{ where }\quad v^*=({\Gamma_{\rvect}^{\varepsilon \dagger}})^\top \begin{pmatrix}
    \operatorname{sign}(\avect^*) \\ 0_{\R^{3K}},
\end{pmatrix}
\end{equation}
where $\Gamma^\varepsilon$ is the operator defined by \eqref{eq:def_gamma}, ${\Gamma_{\rvect}^{\varepsilon \dagger}}$ denotes the pseudoinverse\footnote{In other words, ${\Gamma_{\rvect}^{\varepsilon \dagger}}$ is the unique matrix satisfying all four conditions below:
$$
\Gamma_{\rvect}^{\varepsilon}{\Gamma_{\rvect}^{\varepsilon \dagger}}\Gamma_{\rvect}^{\varepsilon}=\Gamma_{\rvect}^{\varepsilon}, \quad {\Gamma_{\rvect}^{\varepsilon \dagger}}\Gamma_{\rvect}^{\varepsilon}{\Gamma_{\rvect}^{\varepsilon \dagger}}={\Gamma_{\rvect}^{\varepsilon \dagger}}, \quad (\Gamma_{\rvect}^{\varepsilon}{\Gamma_{\rvect}^{\varepsilon \dagger}})^\top = \Gamma_{\rvect}^{\varepsilon}{\Gamma_{\rvect}^{\varepsilon \dagger}}, \quad ({\Gamma_{\rvect}^{\varepsilon \dagger}}\Gamma_{\rvect}^{\varepsilon})^\top = {\Gamma_{\rvect}^{\varepsilon \dagger}}\Gamma_{\rvect}^{\varepsilon}.
$$} of $\Gamma_{\rvect}^{\varepsilon}$.

\begin{proposition}
     The vector $\vvect^*$ defined by \eqref{def:etaV} is the unique solution to:
    \begin{equation}
        \min_{\vvect\in\R^{4K}}\Ltn{\vvect}\quad \text{ such that: } ((\Gamma^\varepsilon)^*\vvect)(\rvect^*_{k})=\operatorname{sign}(a^*_{k}),\; D((\Gamma^\varepsilon)^*\vvect)(\rvect^*_{k})= 0 \quad \forall k\in\dint{1}{K}. 
    \end{equation}
\end{proposition}

\begin{remark}
The precertificate $\eta_V$ interpolates the sign of the true measure at each spike's location, while ensuring a zero derivative at these points.
In the case we consider, the adjoint $(\Gamma^\varepsilon)^*$ takes the following expression: an operator on $\R^{MN}$ that takes its values in $\mathcal C^0(X)$:
\begin{equation}
     (\Gamma^\varepsilon)^*:\begin{array}[t]{ccc}
        \R^{MN} & \longrightarrow & \mathcal C^0(\Reps) \\
        \vvect & \mapsto & \rvect\mapsto \sum_{i=1}^{MN} v_i\gamma_i(\rvect).
    \end{array}
\end{equation}
\end{remark}

A crucial property of the precertificate is its non-degeneracy:
\begin{definition}
    $\eta_V$ is said to be non-degenerate if it verifies:
\begin{enumerate}[(i)]
  \item $\det D^2\eta_V(\rvect_{k}^*)\neq 0$   $\forall k \in\dint{1}{K}$
  \item $|\eta_V(\rvect)| < 1$ if $\rvect \notin \{\rvect^*_{1},\ldots {\rvect^*_{K}}\}$.
\end{enumerate}
\end{definition}

Assuming the kernels $\gamma_{m,n}$ are $C^2$, and that the derivatives up to order 2 vanish at infinity \cite{duval2015exact}, we then have the following theorem:
\begin{theorem}[\cite{duval2015exact}, Theorem 2] \label{th:stable_recovery}
      Assume that $\Gamma^\varepsilon_{\rvect^*}$ has full rank and  $\eta_V$ is non-degenerate.
      Then, there exist two constants $\alpha$, $\lambda_0>0$ such that for all $\lambda$, $\evect$ verifying $0<\lambda\leq \lambda_0$ and $\frac{\Ltn{\evect}}{\lambda}\leq \alpha$, there exists a unique solution $\widetilde\psi$ to \eqref{eq:relaxed_blasso}. Moreover, $\widetilde\psi$ takes the form $\widetilde\psi=\sum_{k=1}^K\widetilde a_k\delta_{\widetilde \rvect_k}$, and if we take $\lambda = \frac{1}{\alpha}\Ltn{\evect}$ we have:
    \begin{equation}
        \infn{\avect^* - \widetilde\avect} = \operatorname{O}(\Ltn{\evect}) \text{ and } \infn{\rvect^* - \widetilde\rvect} = \operatorname{O}(\evect).
    \end{equation} 
\end{theorem}

In other words, if the noise level $\Ltn{\evect}$ is low enough and $\lambda$ is chosen accordingly, there exists a solution to \eqref{eq:relaxed_blasso} in the form of a finite sparse measure that contains as many spikes as the true measure, and this solution converges to the true measure as $\lambda$ and $\Ltn{\evect}$ vanish to zero. \ts{Furthermore, the complete version of the theorem in \cite{duval2015exact} establishes the regularity of the solution with respect to the parameters $(\lambda, \evect)$.}
The idea behind the proof of this theorem and the construction of the precertificate $\eta_V$ is the following: one can use the implicit function theorem to prove the existence of a discrete measure $\psi_{\lambda,\evect}$, composed of the correct number of spikes, that interpolates the sign of the true measure at the source locations with vanishing derivatives. The locations and amplitudes are a function of the noise level and regularization parameter, and it can be shown that $\eta_\lambda\coloneqq\frac{1}{\lambda}(\Gamma^\varepsilon)^*(\xvect_0-\Gamma^\varepsilon\psi_{\lambda,\evect})$ converges to $\eta_V$ when the noise level and $\lambda$ vanish to zero. The assumptions of the theorem then ensure that $\infn{\eta_\lambda}\leq 1$ in the limit, hence the corresponding measure verifies the optimality conditions.

Applying Theorem \ref{th:stable_recovery} is difficult in our case, as the operator $\Gamma^\varepsilon$ is 3-dimensional and depends on a complex geometric relation between the locations of the sources and the positions of the microphones.
However, it is possible to compute numerically the vanishing derivatives pre-certificates $\eta_V$ for a given target measure $\psi^*$ to obtain some insights on the operator's behavior\ts{, and the regularity assumptions on $\eta_V$ are always verified in our case for $\varepsilon > 0$.}  
\begin{figure}[!ht]
    \centering
    \includegraphics[width=\linewidth]{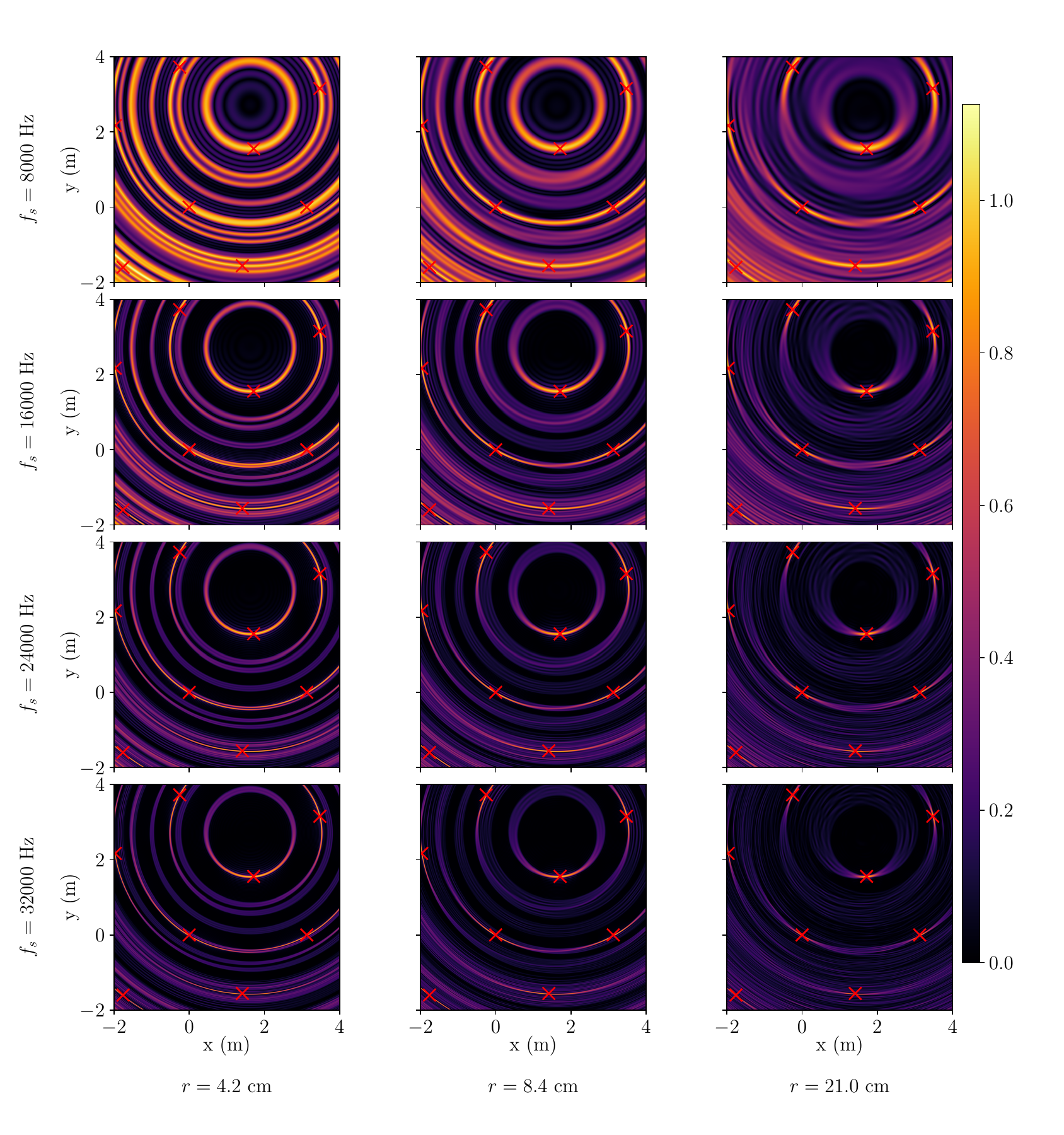}
    \caption{2D plot of the absolute value of the vanishing derivative pre-certificate $\eta_V$ on a section plane parallel to a wall of a room for different sampling frequencies $f_s$ and microphone array radius $r$. The locations of image sources that belong to the plane are marked by red crosses.}
    \label{fig:certif_plot}
\end{figure}
Fig. \ref{fig:certif_plot} represents the values taken by $\eta_V$ on a portion of a plane parallel to a wall that contains several image sources, for varying sampling frequencies. The microphone array's geometry is spherical here, and we proceed to increase the radius of the array, as the array size has an impact on the ability to geometrically locate sources (see Section \ref{sect:IS_inv_exp} for further details on the experimental setup).
Note that the mass of $\eta_V$ is concentrated on spheres that are centered around each microphone, with radii given by the times of arrival of each source to the microphone. Due to the application of the low-pass filter, each sphere is slightly smeared around its true radius, and the location of an image source is given at the intersection of every of its corresponding spheres. As the sampling frequency and the array's radius increase, the total mass of $\eta_V$ becomes more tightly contained around each sphere, ensuring a sharper distribution at the intersection. 
Recall that stable support recovery is guaranteed by Theorem \ref{th:stable_recovery} under an assumption of non-degeneracy of $\eta_V$, \emph{i.e.} $\eta_V(\rvect) < 1$ if $\rvect$ is not a source location. While Fig.~\ref{fig:certif_plot} shows that $\eta_V$ is degenerate at $8~$kHz and for the smallest array radius in this case, it seems to be non-degenerate for greater radii and sampling frequencies. More generally, numerical experiments (see in particular Fig.~\ref{fig:certif_plot}) indicate that, for a fixed room and with measurements from a spherical microphone array, $\eta_V$ is non-degenerate when the sampling frequency and array radius are sufficiently large, and thus that Theorem~\ref{th:stable_recovery} applies.
\subsection{Numerical algorithm}\label{sect:numAlgo}

We implement\footnote{The implementation is written in Python and is publicly available at  \urlstyle{tt}\url{https://github.com/Sprunckt/acoustic-sfw}.} and adapt the Sliding Frank-Wolfe type algorithm introduced in \cite{denoyelle} and initially applied to microscopy in order to solve Problem \eqref{eq:relaxed_blasso}. Let $\psi^{(i)}=\sum_{k=1}^{K^{(i)}}a_k^{(i)}\delta_{\rvect^{(i)}_k}$ be the reconstructed measure at iteration $i$. The reconstruction algorithm consists of two main steps in each iteration: 
\begin{itemize}
    \item[(i)] A new source is located by maximizing the numerical certificate: 
    \begin{equation}\label{eq:def_eta}
\eta^{(i)} : 
\begin{array}[t]{rcl}
\R^3 & \longrightarrow & \R\\
\rvect &\longmapsto& \left((\Gamma^\varepsilon)^*\textbf{res}^{(i)}\right)(\rvect) = \sum_{m,n}\textbf{res}^{(i)}_{m,n}\gamma_{m,n}(\rvect)
\end{array}
\end{equation}
where $\textbf{res}^{(i)} \coloneqq \xvect^{(i)}-\Gamma \psi^{(i)}$ is the residual at iteration $i$.
A stopping criterion can be inferred from the optimality conditions of the BLASSO problem \cite{denoyelle}: $|\eta^{(i)}(\rvect^{(i)}_*)|\leq \lambda$ where $\rvect^{(i)}_*$ 
is the new candidate location. If this criterion is not met, $\rvect^{(i)}_*$ is added to the list of already reconstructed source positions to form $\rvect^{(i+1)}$. 
\item[(ii)] The amplitudes are then updated by solving a non-negative convex LASSO problem:
\begin{equation} 
\avect^{(i+1)} = \underset{\avect\in\R_+^{K^{(i)}+1}}{\operatorname{argmin}}T_\lambda(\avect, \rvect^{(i+1)}). 
\end{equation}
\end{itemize}

The full algorithm is defined in Alg.~\ref{alg:asfw}, and the procedure is explained in detail below.
Step 1 is encompassed by the lines \ref{alg_step:mark1}-\ref{alg_step:mark2} of Alg.~\ref{alg:asfw}. The maximization of $\eta^{(i)}$ is achieved numerically by applying a parallel implementation of the BFGS algorithm \cite{gerber2020optim} to solve the optimization problem. A crucial issue is to provide an accurate guess for initializing BFGS. As the precision of the numerical approximation of the operator increases with the sampling frequency and the size of the microphone array, this becomes increasingly challenging, as illustrated in the precertificate plots of Fig. \ref{fig:certif_plot}.
Due to the spatial extent of the 3D optimization domain, evaluating $\eta^{(i)}$ on a global fine grid would require millions of function evaluations per iteration and is computationally intractable. We consider instead an efficient heuristic to initialize BFGS. As described previously, a true source should be located at the intersection of the spheres centered around each microphone with radii given by the corresponding times of arrival of the source to each microphone.
In order to approximate these times of arrival, we apply a moving average over $3$ samples to the squared residual signal of each microphone and extract the sample with maximal value. We then consider the $8$ microphones with the highest values and build uniform grids with a mean angular spacing of $5^\circ$ on the corresponding spheres. We also mesh the surrounding spheres with radii $\pm5$~cm to obtain a fine grid of approximately $40000$ points. The grid point maximizing $\eta^{(i)}$ is picked as the initial position for off-the-grid optimization.
The optimization problem of Step (ii) is solved at the line \ref{alg_step:mark3} using the Scikit-learn library \cite{pedregosa2011scikit}. \ts{This particular implementation employs a coordinate descent method to solve the LASSO problem}. 

We stop the algorithm when the amplitude of the last estimated source is below a threshold $\alpha_{\min}=0.01$ or if the criterion described in Step 1 is met. 
In order to facilitate the resolution and accelerate the execution, we begin by running the algorithm on a reduced observation vector constructed by limiting the time frame of the signals, \emph{i.e.} only considering the first $j_1$ samples with $j_1<N-1$. When a stopping criterion is reached at line \ref{alg_step:stop1} or \ref{alg_step:stop2}, we extend the RIR if possible, \emph{i.e.} if the time signals have not yet reached $T_{\max}$.
We also increase the number of samples when beginning the iteration at line \ref{alg_step:ext_rir} if $20$ iterations have been effected since the last extension, or if the norm of the residual has sufficiently decreased since the last extension (typically a $70~\%$ decrease from the norm computed at the last extension). The indices $j_1<\ldots <j_L=N-1$ are chosen in order to have a linear progression of the energy, \emph{i.e.} the norm of the vectors $(x_{m,n})_{j_l\leq n <j_{l+1}}^{1\leq m \leq M}$ are roughly constant. The first separating indices $j_l$ are well spaced, while the last indices are more clustered together, as the number of reflections arriving at each microphone increases rapidly.
This extension procedure has the effect of focusing the resolution at first on the closest sources, for which the time of arrivals are usually well separated in the signals. These low-order sources are also the most valuable, as for instance the whole geometric information on the configuration of a cuboid room is embedded in the locations of the original source and the first order image sources.

 Sliding-Frank-Wolfe \cite{denoyelle} introduces a so-called \say{sliding} step at each iteration. The idea is to perform a local descent on both the locations and amplitudes at the end of each iteration $i$ to optimize $T_\lambda(\avect, \rvect)$, using the currently reconstructed measure as initialization. Although this step greatly increases the accuracy of the reconstruction and brings convergence guarantees, the algorithmic complexity explodes for certain rooms in which the number of image sources can reach over a thousand. We thus proceed as in \cite{benard2022fast} and apply the sliding step only once after the very last step. We also delete low amplitude sources before and after optimization, as described in lines \ref{alg_step:end_lines1}-\ref{alg_step:end_lines2}. 
 
 \ts{Note that an additional projection step of the recovered source locations on the set $\Reps$ could be considered in order to handle the operator's singularities. In practice, projecting was not necessary for our experiments when using this initialization strategy, as the considered sources are sufficiently far from the microphones.}
\begin{algorithm}[thp] 
    \caption{Adapted Frank-Wolfe}\label{alg:asfw}
    \begin{algorithmic}[1]
    \REQUIRE{Observation vector $\xvect$, cutting indices $\jvect=(j_1,\ldots,j_L)$}
    \ENSURE{Estimated image-source amplitudes and locations $(\avect_\text{fin}, \rvect_{\text{fin}})$}
    \STATE $\psi^{(0)} \gets 0$
    \STATE $\xvect^{(0)} \gets (x_{m,n})^{1\leq m\leq M}_{1\leq n\leq j_1}$
    \WHILE{$i < i_{\max}$}
    \IF{$\Ltn{\textbf{res}^{(i)}}$ is sufficiently reduced or $i_\text{ext}$ iterations have elapsed since last extension}
    \STATE Extend $\xvect^{(i)}$ if possible \label{alg_step:ext_rir}
    \ENDIF
    \STATE Create an initialization grid $ \mathcal G$ \label{alg_step:mark1}
    \STATE Get $\rvect_{\text{ini}} =\operatorname{argmax}_{\rvect \in \mathcal G}\eta^{(i)}(\rvect)$
    \STATE Get $\rvect_*^{(i)}$ by applying BFGS to $-\eta^{(i)}$ with initial guess $\rvect_{\text{ini}}$\label{alg_step:mark2}
    \IF{$\eta^{(i)}(\rvect_*^{(i)})\leq \lambda$} \label{alg_step:stop1}
    \IF{$\xvect^{(i)}$ can be extended}
    \STATE Extend $\xvect^{(i)}$ and go to next iteration 
    \ELSE
     \STATE Exit the loop 
    \ENDIF
    \ENDIF
    \STATE Get $\rvect^{(i+1)} =\rvect^{(i)}\cup \{\rvect_*^{(i)}\}$
    \STATE Get $\avect^{(i+1)}$ by solving the LASSO problem $\underset{\avect\in\R_+^{K^{(i+1)}}}{\operatorname{min}}T_\lambda(\avect, \rvect^{(i+1)})$ \label{alg_step:mark3} 
    \
    \IF{$a^{(i+1)}_{K^{(i+1)}}<0.01$}  \label{alg_step:stop2}   
    \IF{$\xvect^{(i)}$ can be extended}
    \STATE Extend $\xvect^{(i)}$ and go to next iteration 
    \ELSE
     \STATE Exit the loop 
    \ENDIF
    \ENDIF
    \STATE Delete spikes from $(\avect^{(i+1)}, \rvect^{(i+1)})$ that have amplitudes below $0.01$
    \STATE $i \gets i + 1$
    \ENDWHILE   
    \STATE Delete spikes from $(\avect^{(i)}, \rvect^{(i)})$ that have amplitudes below $0.01$ \label{alg_step:end_lines1}
    \STATE Get $(\avect_\text{fin}, \rvect_{\text{fin}})$ by applying BFGS to $T_\lambda$ with initial guess $(\avect^{(i)}, \rvect^{(i)})$
    \STATE Delete spikes from $(\avect_\text{fin}, \rvect_{\text{fin}})$ that have amplitudes below $0.01$ \label{alg_step:end_lines2}

    \end{algorithmic}
    \end{algorithm}

\section{Numerical experiments}
\label{sect:IS_inv_exp}

We present here some numerical results obtained by applying the algorithm described in the last section.

\subsection{Setup of the simulations}

The experimental setup is the same as in \cite{sprunck2022gridless, sprunck2024fullreversing}: we consider a spherical array of 32 microphones based on the em32 Eigenmike\textsuperscript{\textregistered} (radius $r$=4.2~cm). We generate
200 random cuboid rooms, in which we randomly place the microphone array and sound source, enforcing a minimal separation of $1~$m between the array's center and the source. We also constrain the location of the array's center to be located at least $25~$cm from each wall in order to ensure that every microphone remains in the room. The room lengths and widths in meters are picked uniformly at random in $[2,10]$, while the heights are taken in $[2,5]$. We associate each wall with an absorption coefficient uniformly drawn at random in $[0.01,0.3]$. 
The microphone array is randomly rotated, and a multi-channel discretized room impulse response is generated by applying the operator described in equation \eqref{eq:def_gamma} to the measure composed of the image sources up to order $20$. 
This amounts to truncating the sum in Proposition \ref{prop:IS} to encompass only the source and the image sources that model reflections of order lower or equal to 20.
We set $\kappa= \kappa^\text{lp}$ as defined in \eqref{eq:lowpass_filter} for all experiments.
For each scenario, we choose $N$ as to get signals of duration $T_{\max} = \frac{N-1}{f_s} =50~$ms. Although we use $11521$ image sources to simulate the measurements, only a fraction of these sources has a noticeable impact on the first $50~$ms of each signal. Note that we will only consider the target sources that are in the set $\mathscr{C}$ (as defined in Section \ref{sect:model}) in our metrics. In other words, for evaluation we only look at the target image sources that are in range for every microphone and discard the others. The number of image sources considered in the metrics then ranges from less than $100$ to over $1500$.
$\lambda$ is set to $3.10^{-5}$ in all experiments based on a preliminar study of its impact on reconstruction accuracy. Unless specified otherwise, we consider noiseless simulations, \emph{i.e} a null noise vector $\evect$. For noisy simulations, $\evect$ has a Gaussian distribution $\mathcal N(0,\sigma^2)$, where The standard deviation of the noise $\sigma$ is calculated in the following way relatively to the Peak Signal to Noise Ratio (PSNR, expressed in dB):
\begin{equation}
    \sigma =\max_{m,n} \left| x^*_{m,n}.10^{-\operatorname{PSNR}/20} \right|.
\end{equation}

We then proceed to run the algorithm on every resulting measurement vector and we evaluate the accuracy of the reconstruction.

\subsection{Evaluation metrics}

\subsubsection{Error metrics}

Let $\rvect$ be a target source location, and $\hat{\rvect}$ the estimated source location.
We define the Angular Error (AE) as the angle between the unitary vectors defined by the target and estimated source locations:
\begin{equation}
    \text{AE}(\rvect, \hat{\rvect}) = \arccos\left(\frac{\rvect\cdot\hat{\rvect}}{\Ltn{\rvect}\Ltn{\hat{\rvect}}}\right).
\end{equation}

The Radial Error (RE) is the absolute difference between the norms of the target and estimated source locations:
\begin{equation}
    \text{RE}(\rvect, \hat{\rvect}) = \left|\Ltn{\rvect}-\Ltn{\hat{\rvect}}\right|.
\end{equation}
We will also consider the Euclidean Error (EE):
\begin{equation}
    \text{EE}(\rvect, \hat{\rvect}) = \Ltn{\rvect-\hat{\rvect}}.
\end{equation}

\subsubsection{Recall and precision}
We set radial and angular thresholds at $1~$cm and $2^\circ$ respectively, and we proceed to compute the recall on the recovered sources, that is the proportion of target image sources that were approximated with an error below the thresholds. We then compute the mean radial, angular and Euclidean errors on the sources that are considered as recovered, as well as the mean error on the corresponding amplitudes. We also consider the precision, \emph{i.e.} the proportion of sources in the reconstructed measures that are counted as truly recovered according to the error thresholds.

\subsubsection{Numerical results}

\begin{figure}[t]
    \centering
    \subfigure[Observed and reconstructed signals for two of the microphones]{
    \includegraphics[width=0.5\linewidth]{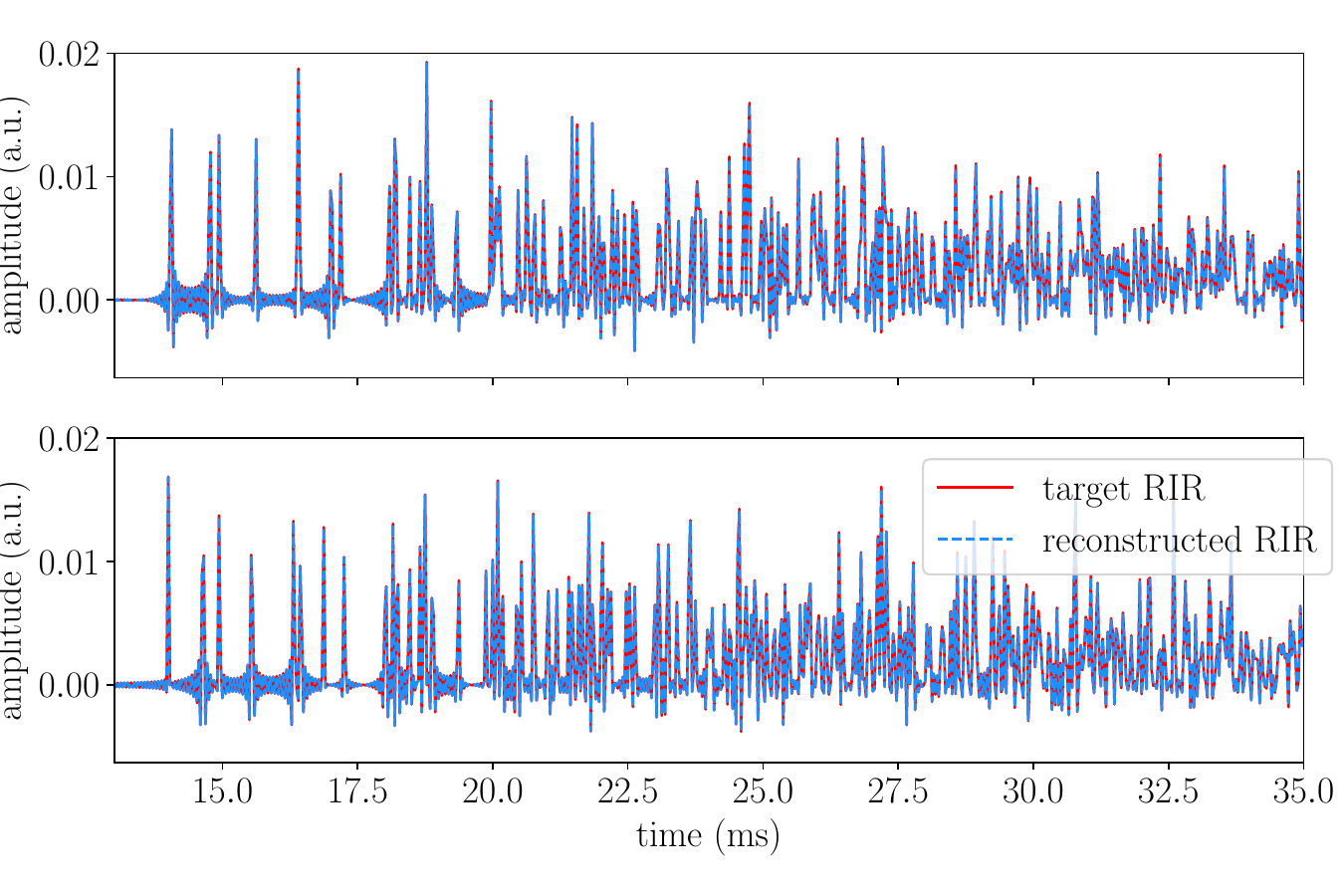}
        \label{fig:reconstr_2d}
    }
    \subfigure[3D plot of the room with its associated true and reconstructed sources]{
    \includegraphics[width=0.43\linewidth]{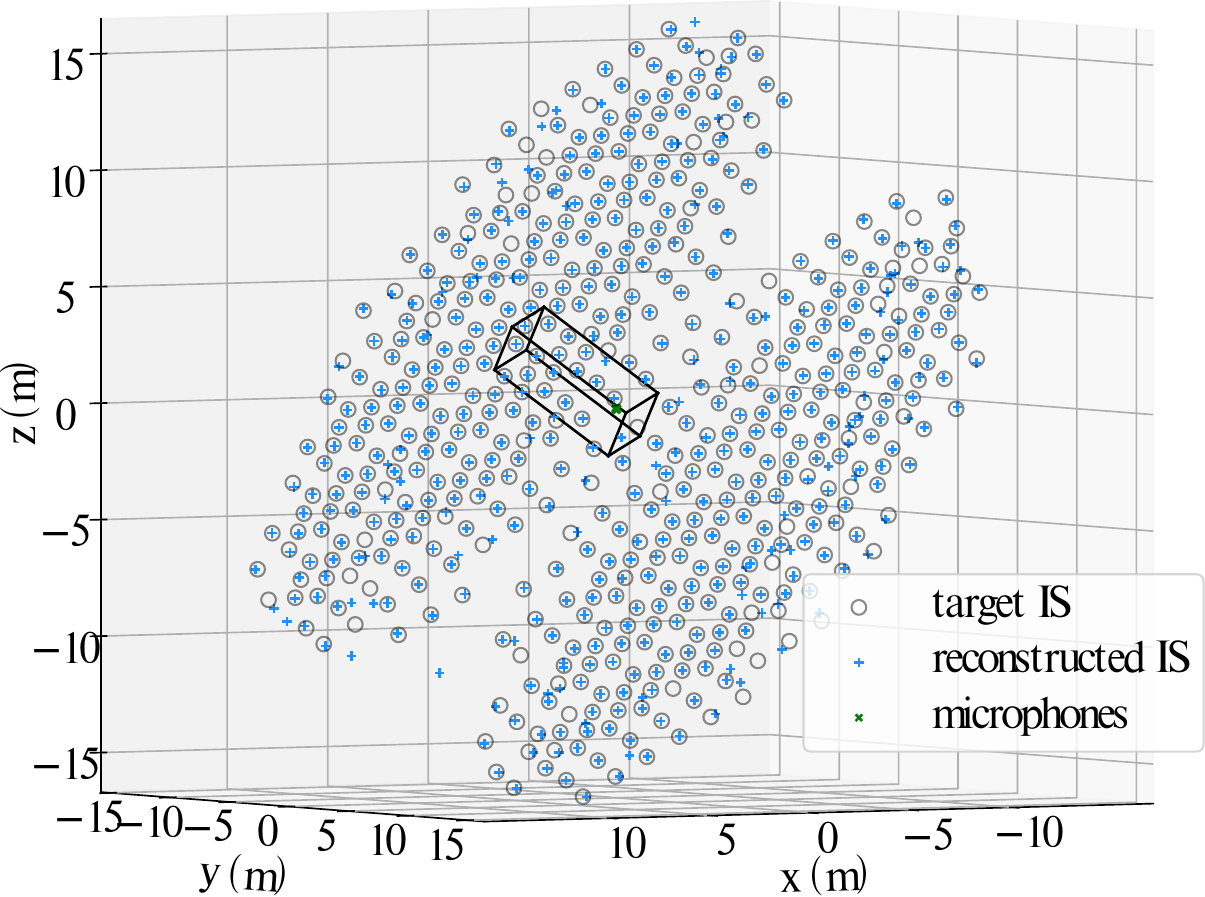}
        \label{fig:reconstr_3d}
    }
    \caption{Reconstruction results for a room of dimensions $6.45\times 2.51\times 2.35$ in meters, resulting in 530 target image sources.  The recall is $93\%$ for radial and angular thresholds of $1~$cm and $2^\circ$, and the mean euclidean error for the recovered sources is $6~$cm. The sampling frequency and microphone array radius are respectively $32~$kHz and $4.2~$cm.}
    \label{fig:reconstr_ex}
\end{figure}

\begin{figure}
    \centering
    \includegraphics[width=0.85\linewidth]{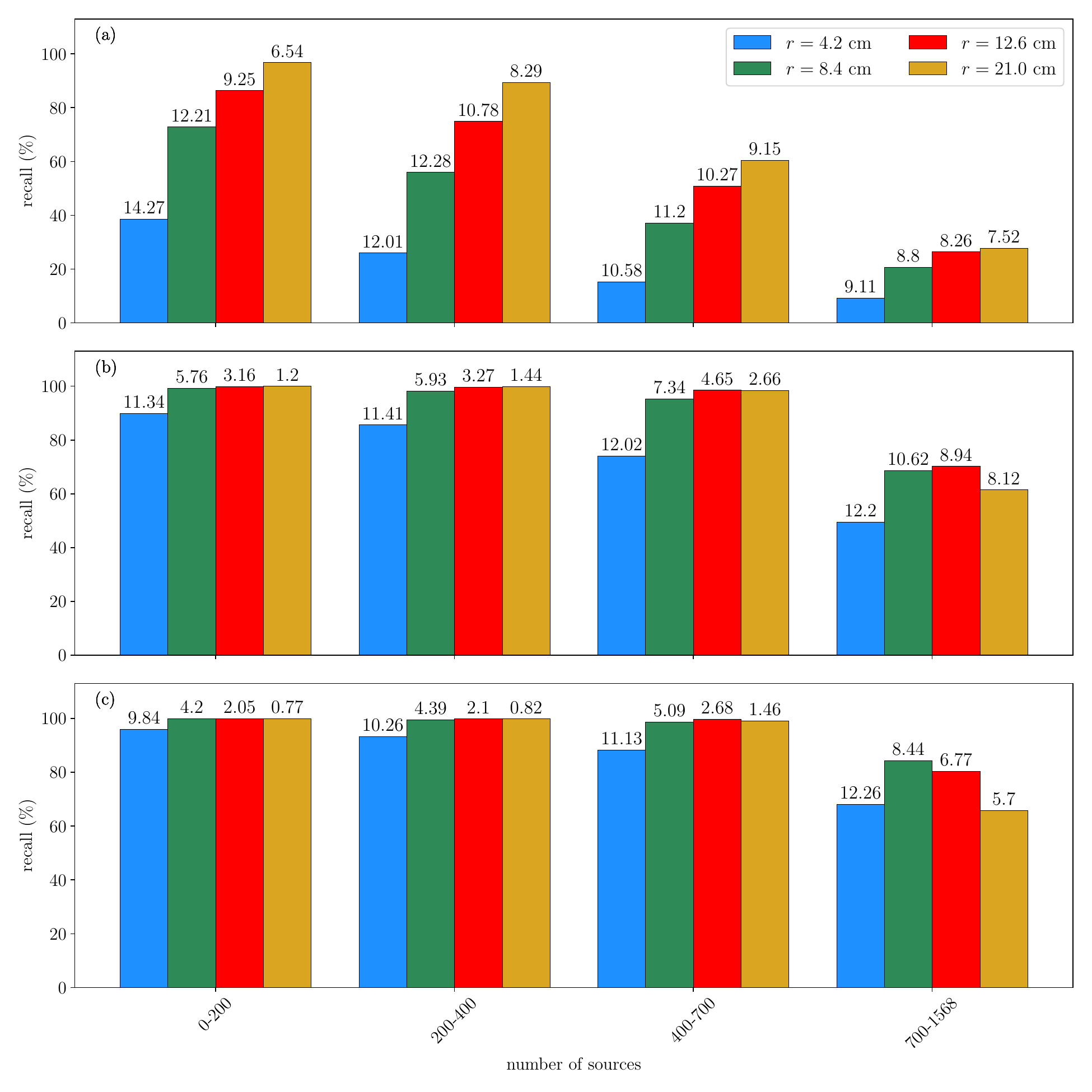}
    \caption{Recall for varying microphone array radii $r$ for $f_s$ taking the values $8~$kHz (a),  $24~$kHz (b), $32~$kHz (c). The room dataset is segmented in four subsets according to the number of target image sources. The mean Euclidean error for the recovered sources is displayed in cm above each bar.}
    \label{fig:recall_grid}
\end{figure}

    \begin{table}[b]
    \centering
         \setlength{\tabcolsep}{4.pt}
    \begin{tabular}{c||c|c|c|c|c|c} 
    \textbf{\# of IS} &
    \textbf{R(\%)} & \textbf{P(\%)} &
    $\overline{\textbf{RE}}$\textbf{(mm)} & $\overline{\textbf{AE}}$\textbf{(°)} & $\overline{\textbf{EE}}$\textbf{(mm)} & $\overline{\textbf{AmE}}$ \\\hline 
0-200     & 89.9 & 80.5 & 0.043 & 0.456 & 113  & 0.034\\
200-400   & 85.6 & 78.6 & 0.062 & 0.454 & 114  & 0.0256\\
400-700    & 74.1 & 67.7 & 0.097  & 0.488 & 120  & 0.022\\
700-1568   & 49.5 & 43.9 & 0.166  & 0.544 & 122 & 0.022\\
    \end{tabular}   
    \caption{Recall (R), precision (P) and mean radial ($\overline{\textrm{RE}}$), angular ($\overline{\textrm{AE}}$), Euclidean ($\overline{\textrm{EE}}$) and amplitude ($\overline{\textrm{AmE}}$) errors among recovered sources as a function of the number of sources, with $r=4.2~$cm and $f_s=24~$kHz.}
    \label{tab:metrics1}
        \end{table}%
\begin{table}\setlength{\tabcolsep}{2.pt}
   \centering
 \begin{tabular}{c||c|c|c|c||c|c|c|c}
 \multicolumn{1}{c||}{}&
    \multicolumn{4}{c||}{\textbf{Noiseless}} & \multicolumn{4}{c}{\textbf{30 dB PSNR}} \\\hline
    \textbf{IS order} & \textbf{R(\%)} &$\overline{\textbf{RE}}$\textbf{(mm)} & $\overline{\textbf{AE}}$\textbf{(°)} & $\overline{\textbf{EE}}$\textbf{(mm)} &  \textbf{R(\%)} &$\overline{\textbf{RE}}$\textbf{(mm)} & $\overline{\textbf{AE}}$\textbf{(°)}& $\overline{\textbf{EE}}$\textbf{(mm)} \\\hline 
    Source & 100 &0.00309&0.0163 &0.996  & 
    100 & 0.0396&0.149 & 8.23 \\
    Order 1 & 99.4 &0.00717&0.0820& 11.7  & 
    98.8 &0.0875&0.346& 37.8 \\
    Order 2 & 98.1 &0.0120&0.151& 27.0  & 
    96.5 &0.131&0.513& 74.2 \\
    Order 3 & 96.0 &0.0207&0.220&44.7  & 
    91.7 &0.175&0.662&112 \\
\end{tabular}
\caption{Recall (R), 
and mean radial ($\overline{\textrm{RE}}$), angular ($\overline{\textrm{AE}}$) and Euclidean ($\overline{\textrm{EE}}$) errors among recovered sources as a function of the image-source order, with $r=4.2~$cm and $f_s=24~$kHz.}
\label{tab:metrics2}
\end{table}

Fig. \ref{fig:reconstr_ex} presents $2$ of the $32$ target and reconstructed time signals for a particular test room, as well as the 3D locations of the sources. 
 Fig. \ref{fig:recall_grid} presents the influence of the microphone array radius $r$ and sampling frequency $f_s$ on the recall. Each bar's height represents the recall, and the corresponding mean Euclidean error is displayed at the top of each bar.  Note that we segmented the room database according to the number of target image-sources, as the number of image-sources varies greatly depending on the size and configuration of the room. In particular, for a limited number of small rooms the number of image sources explodes, increasing the reconstruction's complexity as the echoes become harder to separate in time. We observe that the performance of the algorithm improves as the sampling frequency or the array radius increases, which is expected after the previous observations on the behavior of the certificates (see Fig. \ref{fig:certif_plot}). 
 For the best parameters ($f_s=32~$kHz and $r=21~$cm) we get over $99~$\% recall for the rooms that generate less than $700$ image sources, with a mean Euclidean error under $1.5~$cm. For these rooms, the precision, \emph{i.e.} the proportion of sources in the reconstructed measures that are counted as truly recovered is over $90~$\%. Note that if two sources are reconstructed close to a same target source, only the closest one is counted as a true positive. Table \ref{tab:metrics1} presents the recall (R), precision (P) and mean radial ($\overline{\textrm{RE}}$), angular ($\overline{\textrm{AE}}$), Euclidean ($\overline{\textrm{EE}}$) and amplitude ($\overline{\textrm{AmE}}$) errors amongst the recovered sources for every subset of the room database sources, with $f_s$=24~kHz, $r$=4.2~cm and no noise. The mean Euclidean errors are of the order of a few centimeters. 
 Comparatively, the distance of the sources to each microphone ranges from 1 to over 15 meters, and the error increases with distance due to the compact spherical geometry of the microphone array. Table \ref{tab:metrics2} presents the recall and mean Euclidean error as a function of image-source order when considering the whole room dataset at once, both for the noiseless case and at $30~$dB PSNR. In particular we get a satisfactory $99.4~\%$ recall rate for the first order image sources in the noiseless case, with a mean localization error of $1.17~$cm. Note that if we increase the sampling frequency to $32~$kHz and the array radius to $21~$cm, every first order source is recovered, with a mean Euclidean error of $0.773~$mm (not displayed in the table).

\begin{figure}[t]
    \centering
    \includegraphics[width=0.88\linewidth]{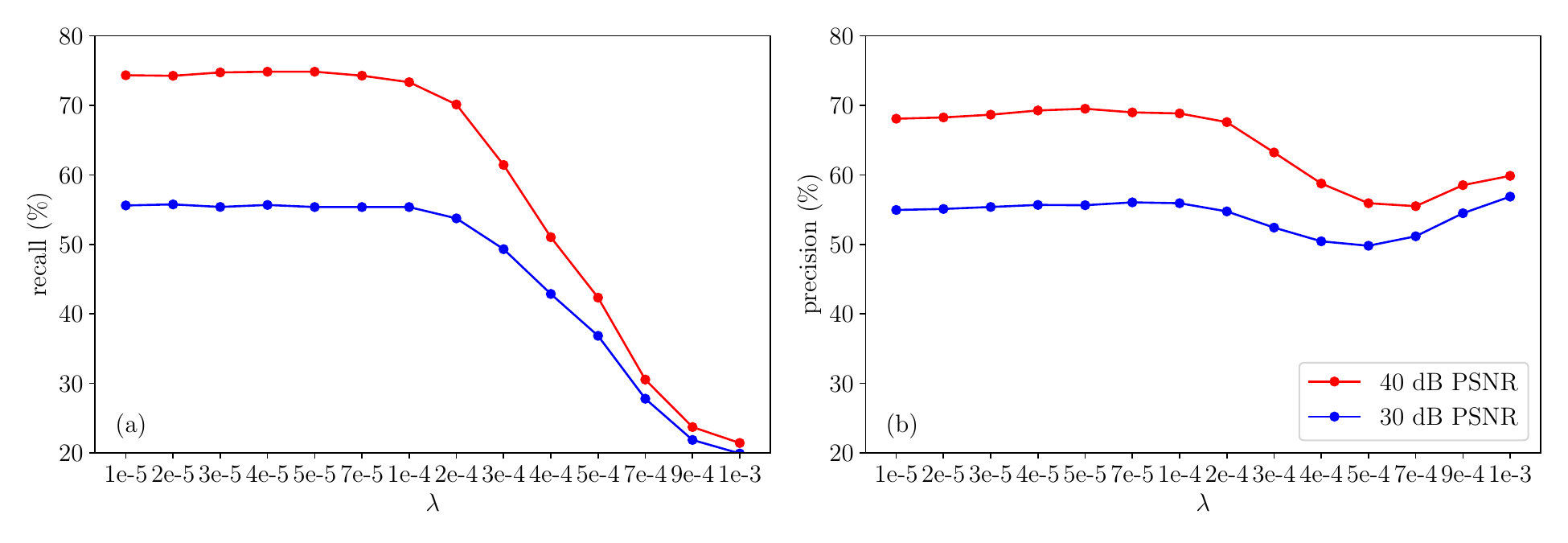}
    \caption{(a) Recall and (b) precision as a function of $\lambda$ at two different PSNRs.}
    \label{fig:lambda}
\end{figure}
\begin{figure}[t]
    \centering
    \includegraphics[width=0.85\linewidth]{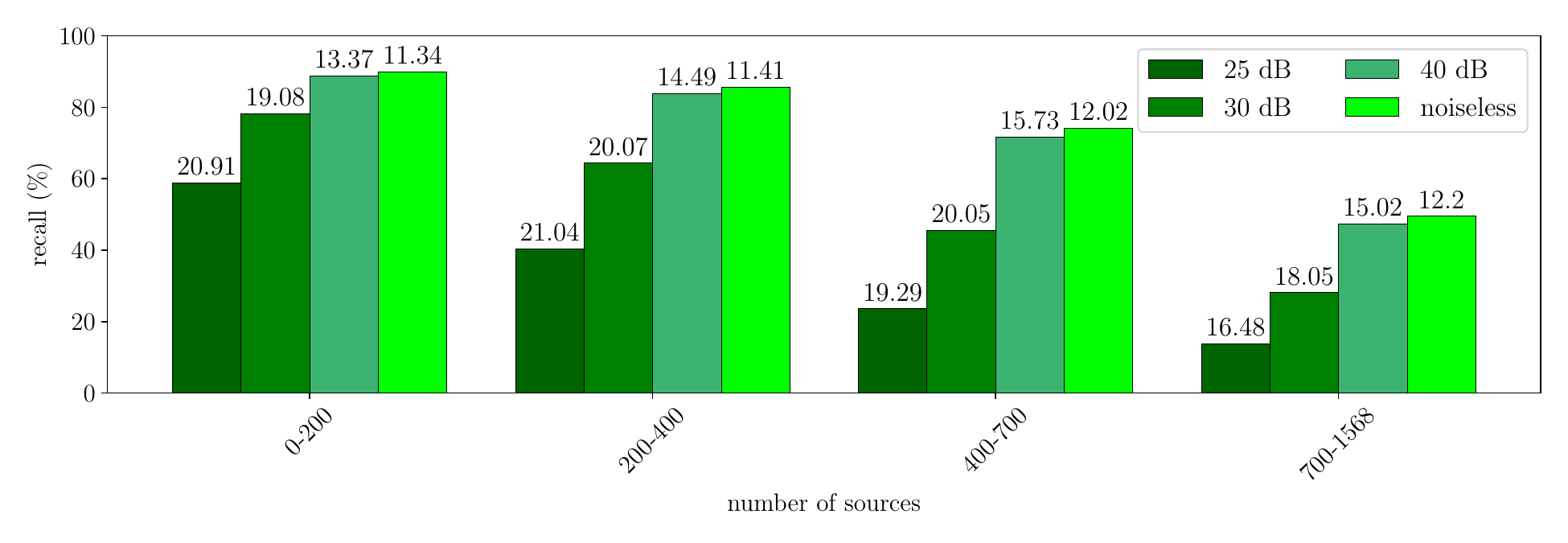}
    \caption{Recall for varying PSNR (in dB) with a microphone array radii $r=4.2~$cm and $f_s=24~$kHz, with the same dataset segmentation as in Fig. \ref{fig:recall_grid}.}
    \label{fig:recall_noise}
\end{figure}

Fig. \ref{fig:lambda} shows how the parameter $\lambda$ affects the quality of the reconstruction at two different noise levels. For low noise levels, the curves are practically flat around the chosen value of the parameter. $\lambda$ could be tuned to increase precision at the cost of recall, especially at high PSNR. However, whilst increasing $\lambda$ can greatly reduce false positives, it does not only reduces the recall for high order image sources, but also for first order sources. For some applications, such as room geometry reconstruction, the locations of first order sources are crucial, which justifies using a low regularization parameter at the expense of additional false positives.

Finally, Fig. \ref{fig:recall_noise} presents the recall obtained for different Peak to Signal Noise Ratios (PSNR). At $40$ PSNR we see little impact on the recall and errors, while the Euclidean errors tend to increase quickly at $30$ PSNR. However, this damages mainly the high order sources, for which the heights of the time-signals peaks are close to the standard deviation of the noise. Indeed, as highlighted in Table \ref{tab:metrics2} even at 30 PSNR we have a $98.8~\%$ recall rate for the first order sources, with an associated mean Euclidean error of $3.78~$cm. These results were obtained for $f_s=24~$kHz and $r=4.2~$cm and might be further improved by increasing the resolution.

\section{Conclusion and open issues}
In this article, we introduced an efficient algorithm to reconstruct the complete set of image sources associated with a room and, indirectly, its shape. We conclude by highlighting several open questions that remain to be explored in order to complete this study, as well as potential directions for future research.

One limitation of our algorithm is that {it requires a large number of initializations} making it computationally expensive. Developing an approach that leverages initializations informed by the configuration of image sources would be a valuable improvement. {Another limitation is that the responses of the walls, source and microphones are idealized. To make the approach applicable to real measurements, the underlying forward model should be extended to include responses that are both frequency- and direction-dependent.}

A more ambitious challenge {is to consider broader classes of room shapes beyond cuboids. One possible route is to extend the underlying image source model to more general polyhedra. One is then faced with the difficult issue of handling \textit{occlusions}, namely, each image source is no longer guaranteed to be visible by all microphones, calling for the use of a more robust data-fit objective. Another route would be to leave the proposed image-source localization paradigm, and directly formulate the inverse problem as a (parameterized) boundary recovery problem based on the wave equation. A shape-optimization algorithm incorporating the concept of Hadamard derivatives could then be employed, which is left as the subject of future work.}


\section*{Acknowledgments}
This work was supported by the French National Research Agency through the DENISE (ANR-20-CE48-0013) and STOIQUES projects.

\section*{Conflicts of interest}
We declare we have no competing interests.

\section*{Ethics statement}
Our study does not involve any ethical concerns as it does not include human participants, animals, or sensitive data.

\section*{Data access statement}
The code required to generate all data used in this study is publicly available at \urlstyle{tt}\url{https://github.com/Sprunckt/acoustic-sfw} and can be accessed for replication and further analysis.


\bibliographystyle{abbrv}
\bibliography{references}

\end{document}